\documentclass[12pt]{extarticle}
\usepackage{amsmath}
\usepackage{amsthm, amssymb, hyperref, color}
\usepackage[shortlabels]{enumitem}
\usepackage{graphicx}
\usepackage[all]{xypic}
\usepackage{makecell}
\usepackage[final]{pdfpages}
\setboolean{@twoside}{false}
\usepackage{pdfpages}
\usepackage{caption}
\usepackage{subcaption}
\usepackage{scalefnt}
\usepackage{verbatim}
\oddsidemargin \evensidemargin
\date{}

\newtheorem{theorem}{Theorem}
\numberwithin{theorem}{section}
\newtheorem{proposition}[theorem]{Proposition}
\newtheorem{lemma}[theorem]{Lemma}
\newtheorem{corollary}[theorem]{Corollary}

\newtheorem{example}[theorem]{Example}

\newcommand{\RR}{\mathbb{R}}
\newcommand{\QQ}{\mathbb{Q}}
\newcommand{\PP}{\mathbb{P}}
\newcommand{\CC}{\mathbb{C}}

\newcommand{\NN}{\mathbb{N}}
\definecolor{g4}{rgb}{0,0.4,0}
\newcommand{\arr}[2]{\begin{array}{#1} #2\end{array}}
\newcommand{\mat}[2]{\left(\!\!\arr{#1}{#2}\!\!\right)}
\newcommand{\utab}[2]{\left(\!\!\!\arr{#1}{#2}\!\!\!\right)}
\newcommand{\M}[1]{#1}
\newcommand{\V}[1]{\M{#1}}

\newcommand{\diag}[1]{\mbox{diag}{#1}}
\newcommand{\trace}[1]{\mbox{trace}{#1}}

\title{\textbf{Distortion Varieties}}
\author{Joe Kileel, Zuzana Kukelova, Tomas Pajdla and Bernd Sturmfels}
\begin{document}
\maketitle

\begin{abstract}
\noindent
The distortion varieties of a given projective variety are parametrized by
duplicating coordinates and multiplying them with monomials.
We study their degrees and defining equations.
Exact formulas are obtained for the case of one-parameter distortions.
These are based on Chow polytopes and Gr\"obner bases.
Multi-parameter distortions are studied using tropical geometry.
The motivation for distortion varieties comes from
multi-view geometry in computer vision. Our theory
furnishes a new framework for formulating and solving
minimal problems for camera models with image distortion. 
\end{abstract}

\section{Introduction}
This article introduces  a construction in algebraic  geometry that is motivated by 
multi-view geometry in computer vision. In that field, one thinks of a camera as
a  linear projection $\PP^3 \dashrightarrow \PP^2$,
and a  model is a projective variety  $X \subset \PP^n$ that represents the
relative positions of two or more such cameras. The data are correspondences of
image points in $\PP^2$. These define a  linear subspace $L \subset \PP^n$, and the task is to 
compute the real points in the intersection $L \cap X$ as fast and accurately as possible.
See \cite[Chapter 9]{HZ-2003} for a textbook introduction.

A model for cameras with image distortion allows for an additional unknown parameter $\lambda$.
Each coordinate of $X$ gets multiplied by a polynomial in $\lambda$
whose coefficients also depend on the data. 
We seek to estimate both $\lambda$ and the point in $X$,
where the data now specify a subspace
$L'$ in a larger projective space $\PP^N$. The distortion variety $X'$
lives in that $\PP^N$, it satisfies ${\rm dim}(X') = {\rm dim}(X)+1$,
and the task is to compute  $L' \cap X'$  in $\PP^N$ fast and accurately.

We illustrate the idea of distortion varieties for the basic scenario
in two-view geometry.

\begin{example} \label{ex:intoP14} \rm
The relative position of two uncalibrated cameras is expressed
by a $3 {\times} 3$-matrix $x = (x_{ij})$ of rank~$2$, known as the {\em fundamental matrix}.
Let $n=8$ and write $F$ for the hypersurface in $\PP^8$ defined by
 the $3 \times 3$-determinant.
Seven (generic) image correspondences in two views determine a line $L$ in $\PP^8$,
and one rapidly computes the three points in $L \cap F$.

The {\em $8$-point radial distortion problem} \cite[Section 7.1.3]{Kukdiss}
is modeled as follows in our setting. We duplicate
the coordinates in the last row and last column of $x$, and we~set
\begin{equation}
\label{eq:intoP14}
\begin{matrix}
\left(
 x_{11} : x_{12} : x_{13} : y_{13}:
 x_{21} : x_{22} : x_{23} : y_{23}:
 x_{31} : y_{31} : x_{32} : y_{32} : x_{33} : y_{33} : z_{33}\right) \,\,\, \,
 = \qquad \quad \qquad
\\  \left(
x_{11}: x_{12}: x_{13} :x_{13} \lambda: 
x_{21}: x_{22}: x_{23} :x_{23} \lambda: 
x_{31}: x_{31} \lambda: 
x_{32}: x_{32} \lambda:
x_{33}: x_{33} \lambda : x_{33} \lambda^2 \right).
\end{matrix}
\end{equation}
Here $N = 14$. The distortion variety
$F'$ is the closure of the set of matrices (\ref{eq:intoP14})
where $x \in F$ and $\lambda \in \CC$.
 The variety $F'$ has dimension $8$ and degree $16$ in $\PP^{14}$,
 whereas $F$ has dimension $7$ and degree $3$ in $\PP^8$.
To estimate both $\lambda$ and the relative camera positions,
we now need eight image correspondences. These data
specify a linear space $L'$ of dimension $6$ in $\PP^{14}$.
The task in the computer vision application is to rapidly compute
the $16$ points in $L' \cap F'$.

The prime ideal of the distortion variety $F'$ is minimally
generated by $18$ polynomials in the $15$ variables.
First, there are $15$ quadratic binomials, namely the
$2 \times 2$-minors of matrix
\begin{equation}
\label{eq:hankel26}
\begin{pmatrix}
x_{13} & x_{23} & x_{31}  & x_{32}  &  x_{33} & y_{33} \\
y_{13} & y_{23} & y_{31}  & y_{32}  &  y_{33} & z_{33}
\end{pmatrix}.
\end{equation}
Note that this matrix has rank $1$ under the substitution (\ref{eq:intoP14}).
Second, there are three cubics
\begin{equation}
\label{eq:threedets}
\begin{matrix}
x_{11} x_{22} x_{33} 
- x_{11} x_{23} x_{32}
- x_{12} x_{21} x_{33}
+ x_{12} x_{23} x_{31}
+ x_{13} x_{21} x_{32}
- x_{13} x_{22} x_{31}, \\
x_{13} x_{22} y_{31} - x_{12} x_{23} y_{31} - x_{13} x_{21} y_{32} +
x_{11} x_{23} y_{32} + x_{12} x_{21} y_{33} - x_{11} x_{22} y_{33}, \\
x_{22} y_{13} y_{31} - x_{12} y_{23} y_{31} - x_{21} y_{13} y_{32}
+ x_{11} y_{23} y_{32} + x_{12} x_{21} z_{33} - x_{11} x_{22} z_{33}.
\end{matrix}
\end{equation}
These three $3 \times 3$-determinants replicate the equation that defines
the original model $F$.
\hfill $\diamondsuit $
\end{example}

This paper is organized as follows.
Section \ref{sec2} introduces the relevant concepts and definitions from
computer vision and algebraic geometry.
We present camera models with image distortion,
with focus on distortions with respect to a
single parameter $\lambda$. 
The resulting distortion varieties $X_{[u]}$ live in the
rational normal scroll $\mathcal{S}_u$, where
 $u = (u_0,u_1,\ldots,u_n)$ is a vector
of non-negative integers. This {\em distortion vector}  indicates that
the coordinate $x_i$ on $\PP^n$ is replicated $u_i$ times when passing to $\PP^N$.
In Example~\ref{ex:intoP14} we have
$u = (0,0,1, 0,0,1, 1,1,2)$ and $\mathcal{S}_u$
is the $9$-dimensional rational normal scroll defined by the $2 \times 2$-minors of (\ref{eq:hankel26}).

Our results on one-parameter distortions of arbitrary varieties are stated and proved in Section \ref{sec3}.
Theorem \ref{thm:degdis} expresses
 the degree of $X_{[u]}$ in terms of the Chow polytope of $X$.
Theorem \ref{thm:mingens} derives
ideal generators for $X_{[u]}$ from a  Gr\"obner basis of $X$.
These results explain what we observed in Example \ref{ex:intoP14},
namely  the degree $16$
and the equations in (\ref{eq:hankel26})-(\ref{eq:threedets}).

Section \ref{sec4} deals with multi-parameter distortions.
We first derive various camera models that are useful
for applications, and we then present the relevant algebraic geometry.

Section \ref{sec5} is concerned with a concrete application
to solving minimal problems in computer vision.
We focus on the distortion variety f+E+$\lambda$ 
of degree $23$ derived in Section~\ref{sec2}.

\section{One-Parameter Distortions}
\label{sec2}

This section has three parts. First, we derive
the relevant camera models from computer vision.
Second, we introduce the distortion varieties $X_{[u]}$ of
an arbitrary projective variety $X$.
And, third, we study the distortion varieties
for the camera models from the first part.

\subsection{Multi-view geometry with image distortion}
\label{sec21}

A {\em perspective camera} in computer vision \cite[pg 158]{HZ-2003}
is a linear projection $\PP^3 \dashrightarrow \PP^2$. The $3 \times 4$-matrix
that represents this map is written as
$\,\M{K} \cdot \mat{c}{\M{R}\,|\,\V{t}} $
where $\M{R} \in  \textup{SO}(3)$, $\V{t} \in \RR^3$,
and $\M{K}$ is an  upper-triangular $3 \times 3$ matrix known as the calibration matrix.
This transforms a point $\V{X}
\in  \PP^3$ from the world Cartesian coordinate  system to  the camera
Cartesian  coordinate  system.  Here, 
we usually normalize homogeneous coordinates
on $\PP^3$ and $\PP^2$ so that the last coordinate equals~$1$.
With this, points in $\RR^3$ map to $\RR^2$ under the action of the camera.

The following  camera  model  was introduced in~\cite[Eqn~3]{Micusik-PAMI-2006}
to deal with image distortions:
\begin{equation}
\quad  \alpha\,    \mat{c}{\M{R}\,|\,\V{t}}    \,    \V{X}    \,\,=\,\,
 \mat{l}{h(\|\M{A}\,U+\V{b}\|)\,(\M{A}\,U+\V{b})\\
 g(\|\M{A}\,U+\V{b}\|)}
 \qquad \mbox{for some}   \,\,\, \alpha \in \RR\backslash \{0\}.
\label{eq:x=(h(AU+b);g)}
\end{equation}
The two functions $h\colon\RR\to\RR$ and  $g\colon\RR\to\RR$ 
represent the distortion. The invertible matrix
 $\M{A} \in  \RR^{2 \times 2}$  and the 
vector $\V{b}  \in \RR^2$ are used to
transform  the image  point  $U \in  \RR^2$  
into the  image Cartesian  coordinate system.
The perspective camera  in the previous paragraph is obtained by
setting  $h=g=1$  and taking the calibration matrix $\M{K}$ to be
the inverse of $\mat{cc}{\M{A}&\V{b}\\0\,  0   &  1}$.

Micusik and Pajdla \cite{Micusik-PAMI-2006} studied  applications to fish eye  lenses  as  well as  catadioptric
cameras. In this context they found that it often suffices to fix $h=1$ and to take
a quadratic polynomial for $g$. For the following derivation we choose
 $g(t) = 1  + \mu t^2 $, where
$\mu$ is an unknown parameter. We also assume that the
calibration matrix has the diagonal form
 $\M{K}   =    \diag{ \bigl[ f,f,1 \bigr]}$.
 If we set $\,\lambda = \mu/f^2\,$ then the model  (\ref{eq:x=(h(AU+b);g)}) simplifies to
 \begin{equation}
\quad  \alpha\,    \mat{c}{\M{R}\,|\,\V{t}}    \,    \V{X}    \,\,=\,\, \M{K}^{-1}  \mat{l}{ \quad U \\
 1 + \lambda \| U \|^2  }
 \qquad \mbox{for some}   \,\,\, \alpha \in \RR\backslash \{0\}.
\label{eq:second}
\end{equation}

Let us now analyze two-view geometry for the model (\ref{eq:second}).
The quantity $\,\lambda = \mu/f^2$ is our distortion parameter.
Throughout the discussion in Section \ref{sec2} there is only one such parameter.
Later, in Section \ref{sec4}, there will be two or more different distortion parameters.

Following \cite[\S 9.6]{HZ-2003} we represent 
two camera matrices 
$  \mat{c}{\M{R}_1 \,|\,\V{t}_1}$  and
$  \mat{c}{\M{R}_2 \,|\,\V{t}_2}$  by
their {\em  essential matrix} $\,\M{E}$. This $3 \times 3$-matrix
has rank $2$ and  satisfies  the  \textit{Demazure equations}.
The equations were first derived in~\cite{Demazure88}; they take the matrix form
$2\,\M{E}\,\M{E}^\top\M{E}-\trace{(\M{E}\,\M{E}^\top)}\M{E}=0$.
For a pair $(U_1,U_2)$ of corresponding points in two images,
the {\em epipolar constraint} now reads
\begin{eqnarray}
0 \,=\, \mat{c}{\M{A} U_2\\1+\mu \|\M{A} U_2\|^2}^{\!\! \top}
\!\!\M{E}\,\mat{c}{\M{A} U_1\\ \! 1+\mu\|\M{A}U_1\|^2}
\,=\, \mat{c}{U_2\\ \! 1+\lambda\|U_2\|^2 \! }^{\!\! \top}\!\!\M{K}^{-\top}\M{E}\,\M{K}^{-1}
\mat{c}{U_1\\1+\lambda\|U_1\|^2 \! }.
\label{eq:0=x2TiKTEiKx1}
\end{eqnarray}
In this way, the essential matrix $\V{E}$ expresses a necessary condition for
two points $U_1$ and $U_2$ in the image planes to be pictures of the same world point.
The {\em fundamental matrix} is obtained from the
essential matrix and the calibration matrix:
\begin{equation}
\label{eq:fundamentalmatrix}
\M{F}  \, =\,
\mat{ccc}{f_{11}&f_{12}&f_{13}\\f_{21}&f_{22}&f_{23}\\f_{31}&f_{32}&f_{33}}
\,\,\, = \,\,\, \M{K}^{-\top}\M{E}\,\M{K}^{-1} .
\end{equation}
Using the coordinates of $U_1 = [u_1,v_1]^\top$ and $U_2 = [u_2,v_2]^\top$,
the epipolar constraint (\ref{eq:0=x2TiKTEiKx1})~is
\begin{small}
\begin{eqnarray}
 0\! &\!=\!&
u_2u_1f_{11}+
u_2v_1f_{12}+
u_2f_{13}+
u_2\|U_1\|^2\lambda f_{13}+
v_2u_1f_{21}+
v_2v_1f_{22}+
v_2f_{23}+ 
v_2\|U_1\|^2\lambda f_{23} + 
 \nonumber\\ & & 
u_1f_{31} +
u_1\|U_2\|^2\lambda f_{31}+
v_1f_{32}+
v_1\|U_2\|^2\lambda f_{32}+
f_{33} +
( \|U_1\|^2 {+}
\|U_2\|^2) \lambda f_{33}
+ \|U_1\|^2\|U_2\|^2\lambda^2 f_{33}.
 \label{eq:0=u1u2f11+u2v1f12+...} 
 \nonumber
\end{eqnarray}
\end{small}
This is a sum of $15$ terms. The corresponding
 monomials in the unknowns form the vector
\begin{equation}
\label{eq:15monomials}
\V{m}^\top = 
\left[
f_{11},f_{12},
f_{13},
f_{13}\lambda,
f_{21},f_{22},
f_{23},f_{23}\lambda,
f_{31},f_{31}\lambda,
f_{32},f_{32}\lambda,
f_{33},f_{33}\lambda,f_{33}\lambda^2
\right].
\end{equation}
The $15$ coefficients are real numbers given by the data. The coefficient vector $c$ is equal to
\begin{small}
$$
\left[u_2u_1,
u_2v_1,
u_2,
u_2\|U_1\|^2 \!\! ,
v_2u_1,
v_2v_1,
v_2, 
v_2\|U_1\|^2\!\!,
u_1,
u_1\|U_2\|^2\!\!,
v_1,
v_1\|U_2\|^2\!\!,
1, 
\|U_1\|^2 {+} \| U_2\|^2\!\!,
\|U_1\|^2\|U_2\|^2
\right]^{\! \top}_.
$$
\end{small}
With this notation, the epipolar constraint given by one point correspondence is simply
\begin{equation}
\label{eq:cdotm}
 \V{c}^\top\V{m}  \,\, = \,\,0 .
 \end{equation}

At this stage we have derived the distortion variety in
Example \ref{ex:intoP14}. Identifying $f_{ij}$
with the variables $x_{ij} $, the vector (\ref{eq:15monomials})
is precisely the same as  that in (\ref{eq:intoP14}).
This is the parametrization of the rational normal scroll
$\mathcal{S}_u$ in $\PP^{14}$ where
$u = (0,0,1,0,0,1,1,1,2)$. The set of fundamental matrices is dense in the
hypersurface $X = \{{\rm det}(\V{F})=0\}$ in $\PP^8$.
Its  distortion variety $X_{[u]}$ has dimension $8$ and degree $16$ in $\PP^{14}$.
Each point correspondence $(U_1,U_2)$ determines a
vector $\V{c}$ and hence a hyperplane in $\PP^{14}$.
The constraint (\ref{eq:cdotm}) means  intersecting $X_{[u]}$ with that hyperplane.
Eight point correspondences determine a $6$-dimensional
linear space in $\PP^{14}$. Intersecting $X_{[u]}$ with that linear subspace
is the same as solving the $8$-point radial distortion problem in \cite[Section 7.1.3]{Kukdiss}.
The expected number of complex solutions is $16$.

\subsection{Scrolls and Distortions}

This subsection introduces the algebro-geometric objects  studied in this paper.
We fix a non-zero vector $u =  (u_0,u_1,\ldots,u_n) \in \NN^{n+1}$ of non-negative
integers, we abbreviate $|u| =  u_0+u_1 +  \cdots + u_n$, and we set  $N =|u|+n$.
The  {\em rational normal scroll} $\,\mathcal{S}_{u} $ is
a  smooth   projective  variety
of dimension $n+1$ and degree $|u|$ in $\PP^N$. It 
has the parametric representation
\begin{equation}
\label{eq:upara} \!\!
\bigl( 
x_0:x_0 \lambda: x_0\lambda^2 : \cdots: x_0 \lambda^{u_0}:
x_1:x_1 \lambda: x_1 \lambda^2 : \cdots: x_1 \lambda^{u_1}:\,\cdots:
x_n:x_n \lambda: \cdots: x_n \lambda^{u_n} \bigr). 
\end{equation}
The coordinates are monomials, so
the scroll  $\mathcal{S}_u$  is also  a toric  variety \cite{CLS}.
Since ${\rm degree}(\mathcal{S}_u) = |u|$ equals ${\rm codim}(\mathcal{S}_u) +1 = N-n+1$,
it is a variety of minimal  degree \cite[Example 1.14]{Har}.   
  
Restriction  to the  coordinates  $(x_0:x_1:\cdots :  x_n)$ defines  a
rational map  $\mathcal{S}_u \dashrightarrow \PP^n$.  This  is a toric
fibration \cite{DR}.  Its fibers are curves parametrized by $\lambda$.
The base  locus is a  coordinate subspace $\PP^n \subset  \PP^N$.  Its
points  have support  on  the last coordinate in  each  of the  $n+1$
groups. For instance,  in Example \ref{ex:123} the base locus
is the $\PP^2$ defined by $ \langle a_0, b_0, b_1, c_0,c_1, c_2 \rangle$
in $\PP^8$.

The prime ideal  of  the scroll $\mathcal{S}_u$  is
generated by the  $2 \times 2$-minors of a $2  \times |u|$-matrix of
unknowns  that is  obtained by  concatenating Hankel  matrices on  the
blocks of unknowns;~see~\cite[Lemma  2.1]{EH}, \cite{Pet}, and Example \ref{ex:123}
below.  For a textbook reference see  \cite[Theorem 19.9]{Har}.

We now consider  an arbitrary projective variety $X$  of dimension $d$
in $\PP^n$.  This is the  underlying model in some
application, such as computer vision.  We define the {\em  distortion variety of level $u$},
denoted $X_{[u]}$,  to be the closure of the preimage of  $X$ under the map
$\mathcal{S}_{u} \dashrightarrow  \PP^n$.  The fibers of  this map are
curves.  The distortion variety $X_{[u]}$ lives in $\PP^N$. It
has dimension~$d+1$.  Points on $X_{[u]}$ represent  points on $X$
whose  coordinates  have  been   distorted  by  an  unknown parameter
$\lambda$.  The parametrization above is  the rule for the distortion.
In other words, $X_{[u]}$ is the closure of the image of the regular map 
$X \times \CC \rightarrow \PP^N$ given by (\ref{eq:upara}).

Each     distortion    variety     represents    a     {\em    minimal
  problem}~\cite{Kukdiss} in polynomial  systems solving.  Data points
define linear constraints  on $\PP^N$, like (\ref{eq:cdotm}).
 Our problem is  to solve $d+1$
such linear equations on $X_{[u]}$. The number of complex solutions is
the  degree  of  $X_{[u]}$. 
A simple bound for that degree is stated in
Proposition \ref{prop:degdis}, and an exact formulas can be found
in Theorem \ref{thm:degdis}.
Of course, in applications  we are primarily interested in  the  real solutions.  

We already saw one example of a distortion variety in Example~\ref{ex:intoP14}.
In the following example, we discuss
some surfaces in $\PP^N$ that arise as distortion varieties of plane curves.

\begin{example} \rm \label{ex:123}
Let  $n   =  2  $   and  $u=(1,2,3)$.   The  rational   normal  scroll
is a $3$-dimensional smooth toric variety in $\PP^8 $. Its implicit equations
are  the $2 \times 2 $-minors of the $2 \times 6$-matrix
\begin{equation}
\label{eq:hankel26b}
\begin{pmatrix}
a_0 & b_0 & b_1 & c_0 & c_1 & c_2 \\
a_1 & b_1 & b_2 & c_1 & c_2 & c_3 
\end{pmatrix}.
\end{equation}
This is the ``concatenated Hankel matrix'' mentioned above.
Its pattern  generalizes to all $u$.

Let $X$  be a general curve  of degree $d$ in  $\PP^2$. The distortion
variety $X_{[u]}$ is  a surface of degree $5d$ in  $\PP^8$.  Its prime
ideal is generated by the  $15$ minors of~(\ref{eq:hankel26b}) together
with  $d+1$ polynomials  of degree  $d$. These  are obtained  from the
ternary form that  defines $X$ by the distortion process in
Theorem \ref{thm:mingens}. For special curves $X$, the degree of $X_{[u]}$ may
drop below $5d$.  For instance, given a line $X = V(\lambda a + \mu b
+ \nu c)$ in $\PP^2$, the  distortion surface $X_{[u]}$ has degree $5$
if $\lambda \not= 0$,  it has degree $4$ if $\lambda = 0$  but $\mu \not= 0$,
and it has degree $3$ if $\lambda=\mu =  0$.  For any curve $X$, the property
${\rm deg}(X_{[u]}) =  5 \cdot {\rm deg}(X)$ holds  after a coordinate
change in $\PP^2$.  If $X = \{p\}$  is a single point in $\PP^2 $ then
$X_{[u]}$  is a  curve in  $\PP^8$. It  has degree  $3$ unless  $p \in
V(c)$.
\hfill $\diamondsuit $
\end{example}
 
\subsection{Back to two-view geometry}
\label{subsec23}

In this subsection we describe several variants of
Example~\ref{ex:intoP14}. These highlight the role
of distortion varieties in two-view geometry.
We fix $n=8$, $N= 14$ and $u = (0,0,1, 0,0,1, 1,1,2)$ as above.
The scroll $\mathcal{S}_u$ is the image of the map (\ref{eq:intoP14})
and its ideal is generated by the $2 \times 2$-minors of (\ref{eq:hankel26}).
Each of the following varieties  
live in the space of $3 \times 3$-matrices $x = (x_{ij})$.

\begin{example}[Essential Matrices] \rm \label{ex:E}
We now write $E$ for the essential variety \cite{Demazure88, FKO}. It has
dimension  $5$ and degree $10$  in $\PP^8$. Its points $x$
are the essential matrices
in (\ref{eq:0=x2TiKTEiKx1}). The  ideal of  $E$  is generated by ten cubics,
 namely ${\rm det}(x)$ and the nine entries of
the  matrix $\,  2 x  x^T  x -  {\rm trace}(xx^T)  x$. The  distortion variety
$E_{[u]}$ has dimension  $6$ and degree $52$ in  $\PP^{14}$. Its ideal
is generated  by $15$ quadrics and  $18$ cubics, derived from  the ten
Demazure cubics. \hfill $\diamondsuit $
\end{example}

\begin{example}[Essential Matrices plus Two Equal Focal Lengths]
   \rm \label{ex:G}
Fix a diagonal  calibration matrix $k = {\rm  diag}(f,f,1)$, where $f$
is  a  new unknown.   We  define  $G$ to  be  the  closure in  $\PP^8$
of the  set of $3 \times 3$-matrices $x$  such that $kxk
\in E$  for some  $f$. To compute the ideal of the variety $G$,
we use the following lines of code in
the computer algebra system {\tt Macaulay2}~\cite{M2}:
\begin{verbatim}
R = QQ[f,x11,x12,x13,x21,x22,x23,x31,x32,x33,y13,y23,y33,y31,y32,z33,t];
X = matrix {{x11,x12,x13},{x21,x22,x23},{x31,x32,x33}}
K = matrix {{f,0,0},{0,f,0},{0,0,1}};
P = K*X*K;
E = minors(1,2*P*transpose(P)*P-trace(P*transpose(P))*P)+ideal(det(P));
G = eliminate({f},saturate(E,ideal(f)))
codim G, degree G, betti mingens G
\end{verbatim}
The output tells us that the  variety $G$ has  dimension $6$  and degree $15$,
and that $G$    is  the complete intersection of  two hypersurfaces in $\PP^8$, namely
the cubic ${\rm  det}(x)$ and the quintic
\begin{equation}
\label{eq:quintic}
 \begin{matrix}
x_{11} x_{13}^3 x_{31} +x_{13}^2 x_{21} x_{23} x_{31}+x_{11}x_{13} x_{23}^2 x_{31}
+x_{21}x_{23}^3 x_{31}-      x_{11} x_{13} x_{31}^3 -x_{21} x_{23}x_{31}^3  + \\
 x_{12} x_{13}^3 x_{32} +x_{13}^2 x_{22} x_{23} x_{32}+x_{12} x_{13} x_{23}^2 x_{32}
+x_{22} x_{23}^3 x_{32}-x_{12} x_{13} x_{31}^2 x_{32} {-} x_{12}^2 x_{13}^2 x_{33} \\
-x_{11} x_{13} x_{31} x_{32}^2- x_{21} x_{23} x_{31} x_{32}^2-x_{12}x_{13} x_{32}^3
{-}x_{22} x_{23} x_{32}^3{-}x_{11}^2 x_{13}^2 x_{33} {-}x_{22} x_{23} x_{31}^2 x_{32} \\
-2 x_{11} x_{13} x_{21}  x_{23} x_{33}-2 x_{12} x_{13} x_{22} x_{23} x_{33}
-x_{21}^2 x_{23}^2 x_{33}   -x_{22}^2 x_{23}^2 x_{33}
   +x_{11}^2 x_{31}^2 x_{33} \\ +x_{21}^2 x_{31}^2 x_{33}
  +2 x_{11}x_{12}x_{31} x_{32}x_{33} + 2 x_{21} x_{22} x_{31} x_{32}x_{33}
  +x_{12}^2 x_{32}^2 x_{33} + x_{22}^2 x_{32}^2 x_{33}.
      \end{matrix}
\end{equation}
The distortion variety $G_{[u]}$ is now computed by the following lines in {\tt Macaulay2}:
\begin{verbatim}
Gu = eliminate({t}, G +
 ideal(y13-x13*t,y23-x23*t,y31-x31*t,y32-x32*t,y33-x33*t,z33-x33*t^2))
codim Gu, degree Gu, betti mingens Gu
\end{verbatim}
We learn that $G_{[u]}$ has dimension $7$ and degree $68$ in $\PP^{14}$.
Modulo the $15$ quadrics for $\mathcal{S}_u$, its
ideal is generated by three cubics,
like those in (\ref{eq:threedets}), and five quintics, derived from (\ref{eq:quintic}).
\hfill $\diamondsuit$
\end{example}

\begin{example}[Essential Matrices plus One Focal Length Unknown] \rm
\label{ex:Gprime}
Let $G'$ denote the $6$-dimensional subvariety of $\PP^8$ defined by
the four maximal minors of the $ 3 {\times} 4$-matrix
\begin{equation}
\label{eq:threebyfour}
 \begin{pmatrix}
     \,x_{11} &  x_{12}  & x_{13} & \,\,x_{21} x_{31}+x_{22} x_{32} + x_{23} x_{33} \\
     \,x_{21} & x_{22} & x_{23}  &  -x_{11} x_{31}-x_{12} x_{32}-x_{13} x_{33} \\
     \,x_{31}  & x_{32} & x_{33} & 0
     \end{pmatrix}.
\end{equation}     
     This  variety has dimension $6$ and degree $9$ in $\PP^8$. It is defined by  one cubic and three quartics. 
The variety $G'$ is similar to $G$ in Example~\ref{ex:G}, but with the
     identity matrix as the calibration matrix for one of the two cameras.
     We can compute $G'$ by running the {\tt Macaulay2} code above but with
      the line     {\tt P = K*X*K} replaced with the line {\tt P = X*K}. This model was studied in
     \cite{Bujnak-etal-ICCV-2009-InProceedings}.
     
  The distortion variety $G'_{[u]}$  has dimension $7$  and degree
$42$ in  $\PP^{14}$.  Modulo the $15$ quadrics that define $\mathcal{S}_u$,  the ideal of $G'_{[u]}$
is minimally generated  by three
cubics and $11$ quartics. \hfill $\diamondsuit$
\end{example}

Table \ref{tab:comparison} summarizes the four models
we discussed in Examples~\ref{ex:intoP14}, \ref{ex:E}, \ref{ex:G} and \ref{ex:Gprime}.
The first column points to a reference in computer vision where this model has been studied.
The last column shows the upper bound for $ {\rm deg}(X_{[u]})$
given in Proposition \ref{prop:degdis}. That bound is not tight in any of our examples.
In the second half of the table we report the same data for
the four models when only  only one of the two cameras
undergoes radial distortion.

\begin{table}[htp]
\small
\centering
\begin{tabular}{|r||c|c|c|c|c|c|c|c|}
\hline
$u = \utab{c}{0,0,1,0,0,1,1,1,2}$ & Ref & $n$ & $N$ & $\!{\rm dim}(X)\!$ & $\! {\rm deg}(X) \!$ & 
$\! {\rm dim}(X_{[u]}) \!$ & $\! {\rm deg}(X_{[u]}) \!$ & 
Prop~\ref{prop:degdis} \\
\hline 
$F$ in Example \ref{ex:intoP14}:  \,\,
$\lambda$+F+$\lambda$   & \cite{Kukdiss}  & 8 & 14 & 7 &  3 & 8 & 16 & 18 \\
\hline
$E$ in Example \ref{ex:E}: \,\,
$\lambda$+E+$\lambda$   & \cite{Kukdiss}  & 8 & 14 & 5 & 10 & 6 & 52 & 60 \\
\hline
$\! G$ in Example \ref{ex:G}: 
$\lambda$f+E+f$\lambda$ & \cite{JKSA}     & 8 & 14 & 6 & 15 & 7 & 68 & 90 \\
\hline
$\! G'$ in Example \ref{ex:Gprime}: 
$\lambda$+E+f$\lambda$ &    & 8 & 14 & 6 & 9 & 7 & 42 & 54 \\
\hline 
\hline
$v = \utab{c}{0,0,1,0,0,1,0,0,1}$ &  Ref  & $n$ & $N$ & $\!{\rm dim}(X)\!$ & $\!{\rm deg}(X) \!$ & 
$\! {\rm dim}(X_{[v]})\! $ & $\! {\rm deg}(X_{[v]})\! $ &  Prop~\ref{prop:degdis} \\
\hline 
$F$ in Example \ref{ex:more}:  \quad
F+$\lambda$   & \cite{KSKA} & 8 & 11 & 7 &  3 & 8 & 8 & 9 \\
\hline
$E$ in Example \ref{ex:more}:  \quad
E+$\lambda$   & \cite{KSKA} & 8 & 11 & 5 & 10 & 6 & 26 & 30 \\
\hline
$G$ in Example \ref{ex:more}:  
f+E+f$\lambda$ &  & 8 & 11 & 6 & 15 & 7 & 37 & 45 \\
\hline
$G'$ in Example \ref{ex:more}:  
E+f$\lambda$ & \cite{KSKA} & 8 & 11 & 6 & 9 & 7 & 19 & 27 \\
\hline
$G''$ in Example \ref{ex:more}:  
f+E+$\lambda$ & & 8 & 11 & 6 & 9 & 7 & 23 & 27 \\
\hline
\end{tabular}
\caption{Dimensions and degrees of two-view models and their radial distortions.}
\label{tab:comparison}
\end{table}

\begin{example} \label{ex:more} \rm
We revisit the four two-view models discussed above,  but with 
distortion vector $v = (0,0,1,0,0,1,0,0,1)$.  Now, $N=11$ and
only one camera is distorted.
The rational normal scroll $\mathcal{S}_{v}$ has codimension $2$ and degree $3$ in
 $\PP^{11}$. Its parametric representation~is
$$
\bigl(\, x_{11} : x_{12} : x_{13}: x_{13}\lambda : 
x_{21} : x_{22} : x_{23} : x_{23}\lambda : 
 x_{31} : x_{32} : x_{33}: x_{33}\lambda  \,\bigr).
$$
The distortion varieties $F_{[v]}$,
$E_{[v]}$, $G_{[v]}$ and $G'_{[v]}$  live in $\PP^{11}$.
Their degrees are shown in the lower half of Table~\ref{tab:comparison}.
For instance, consider the last two rows.
The notation E+f$\lambda$ means that the right
camera has unknown focal length and it is also distorted.

The fifth row refers to another variety  $G''$. 
This is the image of $G'$ under the linear isomorphism that
maps a $3 \times 3$-matrix to its transpose.
 Since $v$ is not a symmetric matrix, unlike $u$, the
variety $G''_{[v]}$ is actually different from $G'_{[v]}$. The descriptor \ f+E+$\lambda$ of $G''_{[v]}$
expresses that the left camera has unknown focal length and the right camera is distorted.
The variety $G''_{[v]}$  has dimension $7$  and degree $23$ in $\PP^{11}$.
In addition to the 
three quadrics $x_{3i} y_{3j} - x_{3j} y_{3i}$
that define $\mathcal{S}_v$, the ideal generators for $G''_{[v]}$ are
two cubics and five quartics. The minimal problem \cite{KSKA, Kukdiss} for this distortion variety is
studied in detail in Section~\ref{sec5}.
 \hfill $\diamondsuit$
\end{example}

\section{Equations and degrees}
\label{sec3}

In this section we express the degree and equations of $X_{[u]}$ in terms of
those of $X$.  Throughout we assume that $X$ is an irreducible variety
of codimension $c $ in $\PP^n$ and the
distortion vector $u \in \NN^{n+1}$ satisfies
$u_0 \leq u_1 \leq \cdots \leq u_n$.
We begin with a general upper bound for the degree.

\begin{proposition} \label{prop:degdis}
Suppose $u_n \geq 1$. 
The degree of the distortion variety satisfies
\begin{equation} \label{eq:upperbound}
{\rm deg}(X_{[u]}) \,\,\leq \, \,
{\rm deg}(X) \cdot (u_{c} + u_{c+1} + \cdots + u_n). 
\end{equation}
This holds with equality if the coordinates are chosen so that $X$ is in general position in $\PP^n$.
\end{proposition}

The upper bound in Proposition~\ref{prop:degdis} is shown for our models in the last column
of Table~\ref{tab:comparison}. This result  will be strengthened in Theorem \ref{thm:degdis} below,
where we give an exact degree formula that
works for all $X$. It is instructive to begin with the two extreme cases.
If $c = 0$ and $X = \PP^n$ then we recover the fact that the scroll
$X_{[u]} = \mathcal{S}_u$ has degree $N-n = u_0 + \cdots + u_n$. If
$c = n$ and $X$ is a general point in $\PP^n$ then
$X_{[u]}$ is a rational normal curve of degree $u_n$.

The following proof, and the subsequent development in this section, assumes
familiarity with two tools from computational algebraic geometry:
the construction of {\em initial ideals} with respect to weight vectors, as in \cite{GBCP},
and the {\em Chow form} of a projective variety  \cite{DS, FKO, GKZ, KSZ}.

\begin{proof}[Proof of Proposition~\ref{prop:degdis}]
Fix $\textup{dim}(X_{[u]}) = n-c+1$ general linear forms on $\PP^{N}$, denoted $\ell_{0}, \ell_1, \ldots, \ell_{n-c}$.
We write their coefficients as the rows
of the  $(n-c+1) \times (N+1)$ matrix
\begin{equation}
\label{eq:genericmatrix}
\begin{bmatrix}
\alpha_{0,0} \, & \, \alpha_{0,1} \, & \, \alpha_{0,2} \,& \, \cdots \, & \, \alpha_{0, N} \\
\alpha_{1,0} \, & \, \alpha_{1,1} \, & \, \alpha_{1,2} \,& \, \cdots \, & \, \alpha_{1, N} \\
\vdots \, & \, \vdots \, & \, \vdots \, & \ddots &  \, \vdots \\
\alpha_{n-c,0} \, & \, \alpha_{n-c,1} \, & \, \alpha_{n-c,2} \,& \, \cdots \, & \, \alpha_{n-c, N}
\end{bmatrix}.
\end{equation}
Here $\alpha_{i,j} \in \CC$.  The degree of $X_{[u]}$ equals $\#\big{(}X_{[u]} \cap V(\ell_{0}, \ldots, \ell_{n-c})\big{)}$.
We shall do this count. Recall that $X_{[u]}$ is the closure of the image of 
the injective map $ X \times \CC  \rightarrow  \PP^{N} $ given~in~(\ref{eq:upara}).
The image of this map is dense in $X_{[u]}$. Its complement is the 
$\PP^n$ consisting  of all points whose
coordinates in each the $n+1$ groups are zero except for the last one.
Since the linear forms $\ell_{i}$ are generic, all points of $X_{[u]} \cap V(\ell_{0}, \ldots, \ell_{n-c})$ lie in this image.
By injectivity of the map, ${\rm deg}(X_{[u]})$ is the 
number of pairs $(x, \lambda) \in X \times \CC$ which map into $X_{[u]} \cap V(\ell_{0}, \ldots, \ell_{n-c})$.

We formulate this condition on $(x, \lambda)$ as follows.  Consider the $(n-c+1) \times (n+1)$ matrix
\begin{equation}
\label{eq:matrixwithlambda}
\!\!\!\!
\begin{bmatrix}
\alpha_{0,0} + \alpha_{0,1} \lambda + \cdots + \alpha_{0,u_{0}}\lambda^{u_0} & \cdots & \cdots & \alpha_{0,u_{0}+\ldots+u_{n-1}+1} +  \cdots + \alpha_{0,N-n}\lambda^{u_{n}} \\
\alpha_{1,0} + \alpha_{1,1} \lambda + \cdots + \alpha_{1,u_{0}}\lambda^{u_0} & \cdots & \cdots & \alpha_{1,u_{0}+\ldots+u_{n-1}+1} +  \ldots + \alpha_{1,N-n}\lambda^{u_{n}} \\
\vdots & \ddots & \ddots & \vdots \\
\alpha_{n-c,0} + \alpha_{n-c,1} \lambda+ \cdots + \alpha_{n-c,u_{0}}\lambda^{u_0} \!\!& \cdots & \cdots 
&\! \! \alpha_{n-c,u_{0}+\ldots+u_{n-1}+1} +  \cdots + \alpha_{n-c,N-n}\lambda^{u_{n}} 
\end{bmatrix}\!.
\end{equation}
We want to count pairs $(x, \lambda) \in \PP^{n} \times \CC$ such that $x \in X$ and $x$ lies in the kernel of this matrix. By genericity of $\ell_{i}$, this matrix has rank $n-c+1$ for all $\lambda \in \CC$.
So for each $\lambda \in \CC$, the kernel
of the matrix  (\ref{eq:matrixwithlambda}) is a linear subspace of dimension $c-1$ in $ \PP^{n}$.  

We conclude that (\ref{eq:matrixwithlambda}) defines a rational curve
in the Grassmannian $\textup{Gr}(\PP^{c-1}, \PP^{n})$. Here the $\alpha_{i,j}$
are fixed generic complex numbers and $\lambda$ is an unknown that parametrizes the curve.
If we take the Grassmannian in its Pl\"ucker embedding then
the degree of our curve is $u_{c} + u_{c+1} + \cdots + u_n$,
which is the largest degree in $\lambda$ of any maximal minor of (\ref{eq:matrixwithlambda}).

At this point we use  the \textit{Chow form} ${\rm Ch}_X$ of the variety $X$.  
Following \cite{DS, GKZ}, this is the defining equation of
an irreducible hypersurface in the Grassmannian $\textup{Gr}(\PP^{c-1}, \PP^{n})$.
Its points are the subspaces that intersect $X$.  The degree of ${\rm Ch}_X$ in Pl\"ucker coordinates
is~$\textup{deg}(X)$. 

We now consider the intersection of our curve with the hypersurface  defined by ${\rm Ch}_X$.
Equivalently, we substitute the maximal minors of (\ref{eq:matrixwithlambda}) into ${\rm Ch}_X$
and we examine the resulting polynomial in $\lambda$. 
Since the matrix entries $\alpha_{i,j}$ in (\ref{eq:genericmatrix}) are generic, 
the curve intersects  the hypersurface of the Chow form ${\rm Ch}_X$
outside its singular locus.
  By B\'ezout's Theorem, 
the number of intersection points 
is bounded above by ${\rm deg}(X) \cdot(u_{c} + u_{c+1} + \cdots + u_n)$.

Each intersection point is non-singular on $V({\rm Ch}_X)$, and so 
the corresponding linear space  intersects the variety $X$
in a unique point $x$. We conclude that the number of desired pairs $(x,\lambda)$
is at most ${\rm deg}(X) \cdot(u_{c} + u_{c+1} + \cdots + u_n)$.
This establishes the upper bound.

For the second assertion, we apply a  general linear change of coordinates to $X$ in $\PP^n$.
Consider the lexicographically last Pl\"ucker coordinate, denoted $p_{c,c+1,\ldots,n}$. 
The monomial $p_{c,c+1,\ldots,n}^{{\rm deg}(X)}$ appears with non-zero
coefficient in  the Chow form ${\rm Ch}_X$. Substituting the maximal minors of (\ref{eq:matrixwithlambda})
into ${\rm Ch}_X$, we obtain a polynomial in $\lambda$ of degree 
${\rm deg}(X) \cdot (u_{c} + u_{c+1} + \cdots + u_n)$. By the genericity hypothesis on
(\ref{eq:genericmatrix}), this polynomial has distinct roots in $\mathbb{C}$.
These represent distinct points in  $X_{[u]} \cap V(\ell_{0}, \ldots, \ell_{n-c})$,
and we conclude that the upper bound is attained.
 \end{proof}

We will now refine the method in the proof above to derive
an exact formula for the degree of $X_{[u]}$ that works in all cases.
The Chow form ${\rm Ch}_X$ is expressed in
primal Pl\"ucker coordinates $p_{i_0,i_1,\ldots,i_{n-c}}$
on $\textup{Gr}(\PP^{c-1}, \PP^{n})$.
The {\em weight} of such a coordinate is the vector
$e_{i_0} + e_{i_1} + \cdots + e_{i_{n-c}}$, and the weight
of a monomial is the sum of the weights of its variables.
The {\em Chow polytope} of $X$ is the convex hull of 
the weights of all Pl\"ucker monomials appearing in ${\rm Ch}_X$; see \cite{KSZ}.

\begin{theorem}
\label{thm:degdis}
The degree of $X_{[u]}$ is the maximum value attained by the linear functional $w \mapsto u \cdot w $ 
on the Chow polytope of $X$. This positive integer can be computed by the formula
\begin{equation}
\label{eq:monosat}
 {\rm degree}(X_{[u]}) \quad = \quad \sum_{j=0}^n\, u_j \cdot {\rm degree}
\bigl(\, {\rm in}_{-u}(X) : \langle x_j \rangle^\infty \,\bigr) ,
\end{equation}
where ${\rm in}_{-u}(X)$ is the initial monomial ideal of  $X$ with respect to a term order 
that refines~$-u$.
\end{theorem}

\begin{proof}
Let $M$ be a monomial ideal in $x_0,x_1,\ldots,x_n$ whose variety is pure
of codimension $c$. Each of its irreducible components is a coordinate subspace
${\rm span}(e_{i_0}, e_{i_1},\ldots, e_{i_{n-c}})$ of $\PP^n$. We write
$\mu_{i_0,i_1,\ldots,i_{n-c}}$ for the multiplicity of $M$
along that coordinate subspace. By \cite[Theorem 2.6]{KSZ}, the Chow form of (the cycle given by) $M$
is the Pl\"ucker monomial $\prod p_{i_0,i_1,\ldots,i_{n-c}}^{\mu_{i_0,i_1,\ldots,i_{n-c}}}$,
and the Chow polytope of $M$ is the point
$\, \sum \mu_{i_0,i_1,\ldots,i_{n-c}} (e_{i_0} + e_{i_1} + \cdots + e_{i_{n-c}})$.
The $j$-th coordinate of that point can be computed from $M$ without
performing a monomial primary decomposition.
Namely, the $j$-th coordinate of the Chow point of $M$  is the degree
of the saturation $M: \langle x_j \rangle^\infty$. This follows from
\cite[Proposition 3.2]{KSZ} and
the proof of \cite[Theorem~3.3]{KSZ}.

We now substitute each maximal minor of the matrix (\ref{eq:matrixwithlambda})
for the corresponding Pl\"ucker coordinate $ p_{i_0,i_1,\ldots,i_{n-c}}$. This results in
a general polynomial of degree $u_{i_0} + u_{i_1} + \cdots + u_{i_{n-c}}$ in the one
unknown $\lambda$. When carrying out this substitution in the Chow form
${\rm Ch}_X$, the highest degree terms do not cancel, and we obtain
a polynomial in $\lambda$  whose degree  is the largest $u$-weight
among all Pl\"ucker monomials in ${\rm Ch}_X$. Equivalently,
this degree in $\lambda$ is the maximum inner product of the vector $u$ with any vertex of 
the Chow polytope of $X$.

One vertex that attains this maximum is the Chow point of the monomial ideal $M = {\rm in}_{-u}(X)$
in the proof of Proposition~\ref{prop:degdis}.  Note that we had chosen one particular term order
to refine the partial order given by $-u$. If we vary that term order then we obtain
all vertices on the face of the Chow polytope supported by $u$.
The saturation formula for the Chow point of
the monomial ideal $M$ in the first paragraph of the proof completes our argument.
\end{proof}

We are now able to characterize when the upper bound in Proposition \ref{prop:degdis} is attained.
Let $c_-$ and $c_+$ be the smallest and largest index 
respectively such that $u_{c_-} = u_{c} = u_{c_+}$.
We define a set $\mathcal{L}_u$ of $n-c+1$ linear forms as follows.
Start with the $n-c_+$ variables
$x_{c_++1}$, $x_{c_++2}$, $\ldots$, $x_n$
and then take $c_+-c+1$ generic linear forms in the variables
$x_{c_-}, x_{c_-+1}, \ldots, x_{c_+}$.
In the case when $u$ has distinct coordinates, $V(\mathcal{L}_u)$
is simply the subspace spanned by
$e_0,e_1,\ldots,e_{n-c}$.

\begin{corollary}
The degree of $X_{[u]}$ is the right hand side of (\ref{eq:upperbound})
if and only if $\,V(\mathcal{L}_u) \cap X = \emptyset$.
\end{corollary}

\begin{proof}
The quantity $\, {\rm deg}(X) \cdot (u_{c} + u_{c+1} + \cdots + u_n)\,$
is the maximal $u$-weight among Pl\"ucker monomials of degree
equal to ${\rm deg}(X)$. The monomials that attain this maximal $u$-weight
are products of ${\rm deg}(X)$ many Pl\"ucker coordinates 
of weight $u_{c} + u_{c+1} + \cdots + u_n$. These 
are precisely the Pl\"ucker coordinates $\,p_{i_0,i_1\ldots,i_{c_+ - c},                                                                                                                                                                                                 \,u_{c_++1}, \ldots, u_n}$, where
$ c_- \leq i_0{<}i_1{<} \cdots {<} i_{c_+ - c} \leq c_+$.

Such monomials are non-zero when evaluated at the subspace $V(\mathcal{L}_u)$.
All other monomials, namely those having smaller $u$-weight, evaluate to zero on $V(\mathcal{L}_u)$.
Hence the Chow form ${\rm Ch}_X$ has terms of degree
$\, {\rm deg}(X) \cdot (u_{c} + u_{c+1} + \cdots + u_n)\,$
if and only if ${\rm Ch}_X$ evaluates to a non-zero constant on $V(\mathcal{L})$ 
if and only if the intersection of  $X$ with $V(\mathcal{L}_u)$ is empty.
\end{proof}

We present two example to illustrate the exact degree formula in Theorem \ref{thm:degdis}.

\begin{example} \rm
Suppose $X$ is a hypersurface in $\mathbb{P}^n$, defined by a 
homogeneous polynomial $\psi(x_0,\ldots,x_n)$ of degree $d$.
Let $\Psi$ be the {\em tropicalization} of $\psi$, with respect to min-plus algebra, as in
\cite{MS}.
Equivalently, $\Psi$ is the support function of the Newton polytope of $f$.
Then
\begin{equation}
\label{eq:tropformula}
 {\rm deg}(X_{[u]} ) \,\, = \,\, d \cdot |u| \, - \,\Psi(u_0,u_1,\ldots,u_n). 
 \end{equation}
For instance, let $n=8, d = 3$ and $\psi$ the determinant of a 
$3 \times 3$-matrix. Hence $X$ is the variety of {\em fundamental matrices}, as in
Example \ref{ex:intoP14}.
The tropicalization of the $3 \times 3$-determinant is
$$ \Psi =
{\rm min}\bigl(
u_{11} {+} u_{22} {+} u_{33},
u_{11} {+} u_{23} {+} u_{32},
u_{12} {+} u_{21} {+} u_{33},
u_{12} {+} u_{23} {+} u_{31},
u_{13} {+} u_{21} {+} u_{32},
u_{13} {+} u_{22} {+} u_{31}
\bigr).
$$
The degree of the distortion variety $X_{[u]}$ equals
 $3 \cdot \sum u_{ij} - \Psi$. This explains the degree $16$ we had observed
 in Example~\ref{ex:intoP14}   for the radial distortion of the fundamental matrices.
 \hfill $\diamondsuit $
 \end{example}

\begin{example} \rm
Let $X$ be the variety of essential matrices with the same distortion vector $u$.
In Example \ref{ex:E}, we found that  $\,{\rm deg}(X_{[u]}) = 52$.
The following {\tt Macaulay2} code verifies this:
\begin{verbatim}
U = {0,0,1,0,0,1,1,1,2};
R = QQ[x11,x12,x13,x21,x22,x23,x31,x32,x33,Weights=>apply(U,i->10-i)];
P = matrix {{x11,x12,x13},{x21,x22,x23},{x31,x32,x33}}
X = minors(1,2*P*transpose(P)*P-trace(P*transpose(P))*P)+ideal(det(P));
M = ideal leadTerm X;
sum apply( 9, i -> U_i * degree(saturate(M,ideal((gens R)_i))) )
\end{verbatim}
Here, ${\tt M}$ is the monomial ideal ${\rm in}_{-u}(X)$,
and the last line is our saturation formula in (\ref{eq:monosat}).
\hfill $\diamondsuit $ \end{example}

\medskip

We next derive the equations that define the distortion variety $X_{[u]}$
from those that define the underlying variety $X$.
Our point of departure is the ideal of the rational normal scroll
$\mathcal{S}_u$. It is generated by the
$\binom{N-n} {2}$ minors of the concatenated Hankel matrix.
The following lemma is well-known and easy to verify using
Buchberger's S-pair criterion; see also \cite{Pet}.

\begin{lemma} \label{lem:GB}
The $2 \times 2$-minors that define the rational normal scroll $\,\mathcal{S}_u$
form a Gr\"obner basis with respect to the diagonal monomial order.
The initial monomial ideal is squarefree. 
\end{lemma}

For instance, in Example~\ref{ex:123}, when $n=2$ and $u = (1,2,3)$,
the initial monomial ideal~is
\begin{equation}
\label{eq:123init}
 \langle a_0 b_1, a_0 b_2, a_0 c_1, a_0 c_2 ,a_0 c_3,
b_0 b_2, b_0 c_1, b_0 c_2, b_0 c_3,
b_1 c_1, b_1 c_2, b_1 c_3,
c_0 c_2, c_0 c_3, c_1 c_3 \rangle. 
\end{equation}
A monomial $m$ is {\em standard} if it does not lie in this initial ideal.
The {\em weight} of a monomial $m$ is the sum of its indices.
Equivalently, the weight of $m$ is the degree in $\lambda$ 
of the monomial in $N+1$ variables that arises from $m$
when substituting in the parametrization of $\mathcal{S}_u$.

\begin{lemma} \label{lem:makestd}
Consider any monomial $x^\nu = x_0^{\nu_0} x_1^{\nu_1} \cdots x_n^{\nu_n}$
of degree $|\nu|$ in the coordinates of $\PP^n$. For any nonnegative integer 
$i \leq \nu \cdot u$ 
there exists a unique monomial
$m$  in the coordinates on $\PP^N$ such that $m$ is standard and maps to
$x^\nu \lambda^i$ under the parametrization of the scroll $\mathcal{S}_u$.
\end{lemma}

\begin{proof}
The polyhedral cone corresponding to the toric variety $\mathcal{S}_u$
consists of all pairs $(\nu,i) \in \RR_{\geq 0}^{n+2}$ with $0 \leq i \leq \nu \cdot u$.
Its lattice points correspond to monomials $x^\nu t^i $ on $\mathcal{S}_u$.
Since the initial ideal in Lemma \ref{lem:GB} is square-free,
the associated regular triangulation of the polytope is unimodular, by \cite[Corollary 8.9]{GBCP}.
Each lattice point $(\nu,i)$ has a unique representation as an $\NN$-linear combination
of generators that span a cone in the triangulation. Equivalently, $x^\nu t^i$ has
a unique representation as a standard monomial in the $N+1$ coordinates on $\PP^N$.
 \end{proof}

\noindent We refer to
the standard monomial  $m$ in Lemma \ref{lem:makestd}
as the {\em $ i$\textup{th} distortion} of the given $\,x^\nu$. 

\begin{example} \rm
In Example~\ref{ex:123} we have $n = 2$, $N = 8$, and $\mathcal{S}_u$
corresponds to the cone over a triangular prism. The lattice points in that cone
are the monomials $x_0^{\nu_0} x_1^{\nu_1} x_2^{\nu_2} t^i$ with 
$ 0 \leq i \leq \nu_0 + 2 \nu_1 +3\nu_2$. Using the ambient coordinates on $\PP^8$,
each  such monomial is written uniquely as $ \,a_0^{\nu_{00}} a_1^{\nu_{01}}
 b_0^{\nu_{10}}  b_1^{\nu_{11}}  b_2^{\nu_{12}} 
  c_0^{\nu_{20}}  c_1^{\nu_{21}}  c_2^{\nu_{22}}  c_3^{\nu_{23}}\,$
  that is not in (\ref{eq:123init}) and satisfies
$ \,  
  \nu_{00} + \nu_{01} = \nu_0,\,
   \nu_{10} + \nu_{11}  + \nu_{12} = \nu_1,\,
   \nu_{20} + \nu_{21} + \nu_{22}  + \nu_{23} = \nu_2,\,
   \nu_{01} + \nu_{11} + 2 \nu_{12} + \nu_{21} + 2 \nu_{22} + 3 \nu_{23} = i $.
   For instance, if $x^\nu = x_0^3 x_1^2 x_2^2$ then its various distortions,
   for $0 \leq i \leq 13$, are the monomials 
      $$ \begin{matrix} 
   a_0^3 b_0^2 c_0^2, \,a_0^3 b_0^2 c_0 c_1, \,a_0^3 b_0^2 c_0 c_2, \,a_0^3 b_0^2 c_0 c_3, \,
   a_0^3 b_0^2 c_1 c_3,   \,   a_0^3 b_0^2 c_2c_3, \,
      a_0^3 b_0^2 c_3^2, \\ a_0^3 b_0 b_1 c_3^2, \,
      a_0^3 b_0 b_2 c_3^2,\, a_0^3 b_1 b_2 c_3^2, \,
      a_0^3 b_2^2 c_3^2, \, a_0^2 a_1 b_2^2 c_3^2, \,a_0 a_1^2 b_2^2 c_3^2, \,
      a_1^3 b_2^2 c_3^2.
      \end{matrix}
      $$
\end{example}

Given any homogeneous polynomial $p$ in the unknowns $x_0,x_1,\ldots,x_n$,
we write $p_{[i]}$ for the polynomial on $\PP^N$ that is
obtained by replacing each monomial in $p$ by its $i$th distortion.

\begin{example} \rm
For the scroll in Example~\ref{ex:123}, the
distortions of the sextic $p = a^6{+}a^2b^2c^2$ are
$$
p_{[0]}  = a_0^6+a_0^2 b_0^2 c_0^2,\,\,
p_{[1]} =  a_0^5 a_1+a_0 a_1 b_0^2 c_0^2\,,\, \ldots,\,\,
p_{[5]} =  a_0 a_1^5+a_1^2 b_1 b_2 c_0^2,\,\,
p_{[6]} =   a_1^6+a_1^2 b_2^2 c_0^2, \ldots
         $$
\end{example}

The following result shows how the equations of $X_{[u]}$
can be read off from those of $X$.

\begin{theorem}
\label{thm:mingens}
The ideal of the distortion variety $X_{[u]}$ is generated by the $\binom{N-n}{2}$ quadrics
that define $\mathcal{S}_u$ together with the
distortions $p_{[i]}$ of the elements $p$ in the reduced Gr\"obner basis 
of $X$ for a term order that refines the weights $-u$.
Hence, the ideal is generated by polynomials whose degree is at most
 the maximal degree of any monomial generator of $\,M = {\rm in}_{-u}(X)$.
\end{theorem}

\begin{proof}
Since $X_{[u]} \subset \mathcal{S}_u$, the binomial quadrics
that define $\mathcal{S}_u$ lie in the ideal $I(X_{[u]})$.  
Also, if $p$ is a polynomial that vanishes on $X$ then all of its distortions $p_{[i]}$
are in $I(X_{[u]})$ because
$$p_{[i]}\big{(}x_{0}, \lambda x_{0}, \ldots, \lambda^{u_{0}}x_{0}, x_{1}, \ldots, \lambda^{u_{n}} x_{n}\big{)} 
\, = \, \lambda^{i} \cdot p(x) \, = \, 0 
\qquad \hbox{for} \,\, \hbox{$\lambda \in \CC$ and $x \in X$}. $$

Conversely, consider any homogeneous polynomial
$F $ in $ I(X_{[u]})$.  It must be shown that $F$ is a polynomial linear combination of 
the specified quadrics and distortion polynomials.  Without loss of generality, we may assume that 
$F$ is standard with respect to the Gr\"obner basis in 
Lemma \ref{lem:GB}, and that each monomial in $F$ has the same weight $i$.
This implies
$$F\big{(}x_{0}, \lambda x_{0}, \ldots, \lambda^{u_{0}}x_{0}, x_{1}, \ldots, \lambda^{u_{n}} x_{n}\big{)} \, = \, \lambda^{i} f(x)$$
for some homogeneous $f \in \CC[x_0, \ldots, x_n]$. 
Since $F \in I(X_{[u]})$, we have $f \in I(X)$.  We write
$$ f \,\,= \,\,h_1 p_1 + h_2 p_2 + \cdots + h_k p_k, $$
where $p_1,p_2, \ldots, p_k$ are in the reduced Gr\"obner basis of $I(X)$ with respect to a term order refining
$-u$, and the multipliers satisfy
$\textup{deg}_{-u}(f) \geq \textup{deg}_{-u}(h_j p_j) = \textup{deg}_{-u}(h_j) + \textup{deg}_{-u}(p_j)$ for
 $j = 1, 2,\ldots, k$.
Since $F = f_{[i]}$, we have $-\textup{deg}_{-u}(f) \geq i$.  
Hence, for each $j $ there exist nonnegative integers $a_j$ and $b_j$
such that $a_j + b_j = i$ and $-\textup{deg}_{-u}(h_j) \geq a_j$ and $-\textup{deg}_{-u}(p_j) \geq b_j$.  The latter inequalities imply that the distortion polynomials
$(h_j)_{[a_j]}$ and $(p_j)_{[b_j]}$ exist. 

Now consider the following polynomial in the coordinates on $\PP^{N}$:
$$ \widetilde{F} \,=\, (h_1)_{[a_1]} \cdot (p_1)_{[b_1]} \,+\, \cdots \,+\, (h_k)_{[a_k]} \cdot (p_k)_{[b_k]}. $$
By construction, $\widetilde{F}$ and $F$ both map to $\lambda^{i} f$ under the parameterization of the scroll $\mathcal{S}_{u}$.
Thus, $\widetilde{F} - F \in I(\mathcal{S}_u)$.  This shows that $F$ is a polynomial linear combination
of generators of $I(\mathcal{S}_u)$ and distortions of
Gr\"obner basis elements $p_1,\ldots,p_k$. This completes the proof.
\end{proof}

We illustrate this result with two examples.

\begin{example} \rm
\label{ex:disteqns1} 
If $X$ is a hypersurface of degree $d \geq 2$ then
the ideal $I(X_{[u]})$ is generated by binomial quadrics and distortion polynomials
of degree $d$. More generally, if the generators of $I(X)$ happen to be a Gr\"obner basis
for $-u$ then the degree of the generators of $I(X_{[u]})$ does not go up.
This happens for all the varieties from computer vision seen in Section~2.
\hfill $\diamondsuit $ \end{example}

In general, however, the maximal degree among the generators of $I(X_{[u]})$ can be 
much larger than that same degree for $I(X)$. This happens for complete intersection
curves in~$\PP^3$:

\begin{example} \rm
Let $X$ be the curve in $\PP^3$ obtained as the intersection of two random  surfaces of degree $4$.
We fix $u = (2,3,4,4)$.
The initial ideal $M = {\rm in}_{-u}(X)$ has $51$ monomial generators. The largest degree is $32$.
We now consider the distortion surface $X_{[u]}$ in $\PP^{12}$.
The ideal of $I(X_{[u]})$ is minimally generated by $133$ polynomials. The largest degree is~$32$.
\hfill $\diamondsuit $ \end{example}

\section{Multi-parameter Distortions}
\label{sec4}

In this section we study multi-parameter
distortions of a given projective variety $X \subset \PP^{n}$. 
Now, $\lambda = (\lambda_1, \ldots, \lambda_r)$ is a vector of $r$ parameters,
and $u = (u_0,\ldots,u_n)$ where $u_i = \{u_{i,1}, u_{i,2}, \ldots,u_{i,s_i}\}$
 is an arbitrary finite subset of $\NN^r$. Each point $u_{i,j}$
 represents a  monomial in the $r$ parameters, denoted $\lambda^{u_{i,j}}$.    We set
$|u| = \sum_{i=0}^n |u_i| = \sum_{i=0}^n s_i$ and $N = |u| -1$.  
The role of the scroll is played by a toric variety $\mathcal{C}_u$
of dimension $n+r$ in $\PP^N$ that is usually not smooth.
Generalizing (\ref{eq:upara}), we define the {\em Cayley variety} $\,\mathcal{C}_u$ 
in $\PP^N$ by the parametrization
\begin{equation}
\label{eq:uparamulti}
\bigl(
 x_0 \lambda^{u_{0,1}} : x_0 \lambda^{u_{0,2}}: \cdots: x_0 \lambda^{u_{0,s_0}} \,:\,
 x_1 \lambda^{u_{1,1}}:   \cdots: x_1 \lambda^{u_{1,s_1}} \,:\, \cdots \,:\,
x_r \lambda^{u_{r,1}}: \cdots: x_r \lambda^{u_{r,s_r}} \bigr).
\end{equation}
The name was chosen because $\mathcal{C}_u$ is the toric variety
associated with the Cayley configuration of the configuration $u$.
Its convex hull is the  {\em Cayley polytope}; see \cite[\S 3]{DR} and \cite[Def.~4.6.1]{MS}.
 
The {\rm distortion variety} $X_{[u]}$ is defined as the closure of the set of all points
(\ref{eq:uparamulti}) in $\PP^N$ where $x \in X$ and $\lambda \in (\CC^*)^r$.
Hence $X_{[u]}$ is a subvariety of the Cayley variety $\mathcal{C}_u$,
typically of dimension $d+r$ where $d = {\rm dim}(X)$.
 Note that, even in the single-parameter setting $(r=1)$, we have generalized our construction, by permitting $u_{i}$
to not be an initial segment of $\NN$.

\begin{example} \rm Let $r =n=2$, $u_0 =  \{(0,0),(0,1)\}$,
$u_1  = \{(0,0),(1,0)\}$, $u_2 = \{(2,2),(1,1)\}$.
The Cayley variety $\mathcal{C}_u$ 
is the singular hypersurface in $\PP^5$
defined by $a_0 b_0 c_0 - a_1 b_1 c_1$.
Let $X$ be the conic in $\PP^2$ given by $x_0^2 + x_1^2 - x_2^2$.
The distortion variety $X_{[u]}$ is a threefold of degree $10$. Its ideal is
$\langle a_0 b_0 c_0-a_1 b_1 c_1, a_0^2 c_0^2+b_0^2 c_0^2-c_1^4,
     a_0^2 a_1 b_1 c_0+a_1 b_0^2 b_1 c_0-a_0 b_0 c_1^3,
     a_0^2 a_1^2 b_1^2+a_1^2 b_0^2 b_1^2-a_0^2 b_0^2 c_1^2 \rangle$.~$\diamondsuit$
\end{example}

\subsection{Two views with two or four distortion parameters}
\label{subsec41}

We now present some  motivating examples from computer vision.
Multi-dimensional distortions arise when several cameras have different
unknown radial distortions, or when the distortion function
$g(t) = 1+\mu t^2$ in (\ref{eq:x=(h(AU+b);g)})--(\ref{eq:second})
is replaced by a polynomial of higher degree.

We return to the setting of Section \ref{sec2}, and we introduce two
 distinct distortion parameters  $\lambda_{1}$ and $\lambda_{2}$,
 one for each of the two cameras.
 The role of  the equation (\ref{eq:0=x2TiKTEiKx1}) is played by
  \begin{eqnarray}
0 \! &=&
\mat{c}{U_2\\1+\lambda_{2}\|U_2\|^2}^\top\!\!\begin{bmatrix} x_{11} & 
x_{12} & x_{13} \\ x_{21} & x_{22} & x_{23} \\ x_{31} & x_{32} & x_{33} \end{bmatrix}
\mat{c}{U_1\\1+\lambda_{1}\|U_1\|^2}.
\label{eq:2param}
\end{eqnarray}
Just like in (\ref{eq:cdotm}), this translates into one linear equation
 $\V{c}^\top\V{m}   = 0 $, where now
$ \V{m}^\top =
[
x_{11},x_{12}, $ $
x_{13}, \lambda_{1}x_{13},
x_{21},x_{22},
x_{23}, \lambda_{1}x_{23},
x_{31},x_{31}\lambda_{2},
x_{32},x_{32}\lambda_{2},
x_{33},x_{33}\lambda_{2},x_{33}\lambda_{1},x_{33}\lambda_{1}\lambda_{2} ]$ and
$\V{c}^\top $ equals
\begin{small} 
$ \left[u_2u_1, \!
u_2v_1, \!
u_2, \!
u_2\|U_1\|^2 \!,
v_2u_1, \!
v_2v_1, \!
v_2,  
v_2\|U_1\|^2 \!,
u_1, 
u_1\|U_2\|^2 \!,
v_1, 
v_1\|U_2\|^2, 
1,\|U_1\|^2 \!,
\|U_2\|^2 \!,
\|U_1\|^2\|U_2\|^2
\right]
$.
\end{small}

Here $\V{c}$ is a real vector of data, whereas
$\lambda = (\lambda_1,\lambda_2)$ and  $x = (x_{ij})$ comprise $11$ unknowns.
The vector $\V{m}$ is a monomial parametrization of the form (\ref{eq:uparamulti}).
The corresponding configuration $u$ is given by
$ u_{11} =  u_{12} = u_{21} = u_{22} = \{(0,0)\}, u_{13} =  u_{23} = \{(0,0), (1,0)\}, u_{31} = u_{32} = \{(0,0), (0,1)\}, u_{33} = \{ (0,0), (1,0), (0,1), (1,1)\}$.
The Cayley variety $\mathcal{C}_u $ lives in $\PP^{15}$. It has
dimension $10$ and degree $10$.
Its toric ideal is generated by $11$ quadratic binomials.

Let $X \subset \PP^{8}$ be one of the two-view models $F$, $E$, $G$, or $G'$ in 
Subsection \ref{subsec23}.  The following table concerns the
distortion varieties $X_{[u]}$ in $\PP^{15}$. It is an extension of Table \ref{tab:comparison}.

\begin{table}[htp]
\small
\centering
\begin{tabular}{|r||c|c|c|c|c|c|c|c|}
\hline
 &  $\!{\rm dim}(X)$, &  
$\! {\rm dim}(X_{[u]}) \!$ & $\! {\rm deg}(X_{[u]}) \!$ & 
Prop~\ref{prop:degdis}& \# ideal gens of \\
& $\! {\rm deg}(X) \!$  & & & iterated &deg 2, 3, 4, 5\\
\hline 
$F$ in Example \ref{ex:intoP14}:  \,\,
$\lambda_{1}$+F+$\lambda_{2}$    & 7,  3 & 9  & 24 & 36 & 11, \,4, \,0, 0 \\
\hline
$E$ in Example \ref{ex:E}: \,\,
$\lambda_{1}$+E+$\lambda_{2}$   & 5, 10 & 7 & 76& 120 & 11, 20, 0, 0\ \\
\hline
$\! G$ in Example \ref{ex:G}: 
$\lambda_{1}$f+E+f$\lambda_{2}$ & 6, 15 & 8 & 104 & 180 & \, 11,  \,\,4, \,0, 4 \ \\
\hline
$\! G'$ in Example \ref{ex:Gprime}: 
$\lambda_{1}$+E+f$\lambda_{2}$ &   6, 9 & 8 & 56 & 108 & \, \,11, 4, 15, 0 \ \\
\hline 
\end{tabular}
\caption{\small{Dims, degrees, mingens of two-view models and their two-parameter radial distortions.}}
\label{tab:2param}
\end{table}

On each $X_{[u]}$ we consider linear systems of equations 
$\V{c}^\top\V{m}   = 0 $  that arise from point correspondences.
For a minimal problem, the number of such epipolar constraints is ${\rm dim}(X_{[u]})$,
and the expected number of its complex solutions is ${\rm deg}(X_{[u]})$.
The last column summarizes the number of minimal generators of
the ideal of $X_{[u]}$. For instance, the variety $X_{[u]} =E_{[u]} $ for
essential matrices is defined by $11$ quadrics (from $\mathcal{C}_u$),
$20$ cubics, $0$ quartics and $0$ quintics.
If we add $7$ general linear equations to these then we have a system
with $76$ solutions in $\PP^{15}$. The penultimate column of Table \ref{tab:2param} 
gives an upper bound on ${\rm deg}(X_{[u]})$ that is obtained
by applying Proposition~\ref{prop:degdis} twice, after
decomposing $u$ into two one-parameter distortions.

We next discuss four-parameter distortions for two cameras.  These are based on the following model for 
 epipolar constraints, which is a higher-order version of equation~(\ref{eq:2param}):
\begin{eqnarray}
0 \! &=&
\mat{c}{U_2\\1+\lambda_{2}\|U_2\|^2+\mu_{2}\|U_2\|^4}^\top \begin{bmatrix} x_{11} & x_{12} & x_{13} \\ x_{21} & x_{22} & x_{23} \\ x_{31} & x_{32} & x_{33} \end{bmatrix}
\mat{c}{U_1\\1+\lambda_{1}\|U_1\|^2+\mu_{1}\|U_1\|^4}.
\label{eq:4param}
\end{eqnarray}
As before, the $3 \times 3$-matrix $x = (x_{ij})$  belongs to a
two-view camera model $E$, $F$, $G$ or $G'$.
We rewrite (\ref{eq:4param}) 
as the inner product  $\V{c}^\top\V{m}   = 0 $  of two vectors,
where $\V{c}$ records the data and $\V{m}$ is a parametrization for the 
distortion variety. We now have $n=9, r = 4$ and $|u| = 25$.
  The configurations in $\NN^{4}$ that furnish the degrees for this four-parameter distortion are
\begin{small}
$$
\begin{matrix}
 u_{11} =  u_{12} = u_{21} = u_{22} = \{{\bf 0}\}, \\ u_{13} =  u_{23} = \{{\bf 0}, (1,0,0,0), 
 (0{,}0{,}1{,}0)\}, u_{31} = u_{32} = \{{\bf 0}, (0,1,0,0), (0,0,0,1)\}, \\ \!\!\!\!
u_{33} = \{ {\bf 0}, (1,0,0,0), (0,1,0,0), (0,0,1,0), (0,0,0,1), (1,0,1,0),(1,0,0,1),(0,1,1,0),(0,1,0,1)\}.
\end{matrix}
$$
\end{small}

\noindent Each of the resulting distortion varieties $X_{[u]}$ lives in $\PP^{24}$ and satisfies
 ${\rm dim}(X_{[u]}) = {\rm dim}(X) + 4$.
As before, we may compute the prime ideals for these distortion varieties by elimination,
for instance in \texttt{Macaulay2}. From this, we
 obtain the information displayed in Table \ref{tab:4param}.

\begin{table}[htp]
\small
\centering
\begin{tabular}{|r||c|c|c|c|c|c|}
\hline
 &dimension&degree&quadrics & cubics & quartics & quintics \\
\hline
$F$ in Example \ref{ex:intoP14}:  \,\,
$\lambda_{1}\mu_{1}$+F+$\lambda_{2}\mu_{2}$ & 11 & 115 & 51 & 9 & & \\
\hline
$E$ in Example \ref{ex:E}: \,\,
$\lambda_{1}\mu_{1}$+E+$\lambda_{2}\mu_{2}$ & 9 & 354 & 51 & 34 & & \\
\hline
$\! G$ in Example \ref{ex:G}:~$\lambda_{1}\mu_{1}$f+E+f$\lambda_{2}\mu_{2}\!$
& 10 & 245 & 51 & 9 & 42 & \\
\hline
$\! G'$ in Example \ref{ex:Gprime}:~$\lambda_{1}\mu_{1}$+E+f$\lambda_{2}\mu_{2}\!$ & 10 & 475 & 51 & 9 & & 9 \\
\hline
\end{tabular}
\caption{\small{Dimension, degrees, number of minimal generators for four-parameter radial distortions.}}
\label{tab:4param}
\end{table}

In each case, the 51 quadrics are binomials that define the ambient 
Cayley variety $\mathcal{C}_u $ in $ \PP^{24}$.
The  minimal problems are now more challenging than those in Tables
 \ref{tab:comparison} and \ref{tab:2param}.
 For instance, to recover the essential matrix along with four
 distortion parameters from  $9$ general point correspondences,
   we must solve a polynomial system that has $354$ complex solutions.

\subsection{Iterated distortions and their tropicalization}

In what follows we take a few steps towards a geometric
theory of multi-parameter distortions.
We begin with the observation that multi-parameter distortions arising in practise,
including those in Subsection~\ref{subsec41},
will often have an inductive structure. Such a structure allows us to decompose them
as successive one-parameter distortions where the degrees form an initial segment of 
the non-negative integers $\NN$. In that case the results of
Section \ref{sec2} can be applied iteratively. The following proposition
characterizes when this is possible.
For $u_i \subset \NN^r$ and $k < r$, we write $u_{i}\rvert_{\NN^{k}} \subset \NN^k$
for the projection of the set $u_{i}$  onto the first $k$ coordinates.

\begin{proposition} \label{prop:successive}
Let $u = (u_0, \ldots, u_n)$ be a sequence of finite nonempty subsets of $\NN^r$. 
The multi-parameter distortion with respect to $u$ in $\lambda_{1}, \ldots, \lambda_{r}$
is a succession of one-parameter distortions by initial segments, in $\lambda_{1}$, 
then $\lambda_{2}$, and so on, if and only if
each fiber~of the maps $u_{i}\rvert_{\NN^{k}} \twoheadrightarrow u_{i}\rvert_{\NN^{k-1}}$
becomes an initial segment of $\,\NN$ when  projected onto 
 the $k^{\textup{th}}$ coordinate. This condition holds when
each $u_i$ is an order ideal in the poset $\NN^r$, 
with coordinate-wise order.
\end{proposition}

\begin{proof}
We show this for $r=2$.
 The general case is similar but notationally more cumbersome.
 The two-parameter distortion given by a sequence
  $u$ decomposes into two one-parameter distortions if and only if there exist vectors
 $v = (v_0,\ldots,v_n) \in \NN^{n+1}$ and $w = (w_0,\ldots,w_n) \in \NN^{v_0+1} \oplus \cdots \oplus \NN^{v_n+1}$
such that $\, u_i = \{(s,t) : 0 \leq s \leq v_i \,$ and $\, 0 \leq t \leq w_{is} \} \,$ for $i=0,1,\ldots,n$.
This means that both the Cayley variety and any distortion subvariety decomposes as follows:
\begin{equation}
\label{eq:CSXdecomp}
 \mathcal{C}_u \,=\, (\mathcal{S}_{v})_{[w]} \quad \hbox{and} \quad
X_{[u]} \,=\, ({X}_{[v]})_{[w]}. 
\end{equation}
The segment $[0,v_i]$ in $\NN$ is the unique fiber of the map
$u_{i}\rvert_{\NN^{1}} \twoheadrightarrow u_{i}\rvert_{\NN^{0}} = \{0\}$.
The fiber of
$u_{i}\rvert_{\NN^{2}} \twoheadrightarrow u_{i}\rvert_{\NN^{1}} = [0,v_i]$
over an integer $s$ is the segment $[0,w_{is}]$ in $\NN$.
Thus the stated condition on  fibers is equivalent to the existence of the
non-negative integers $v_i$ and $w_{is}$. For the second claim, we note that the set
 $u_i$ is an order ideal in $\NN^2$ precisely when
$w_{i0} \geq w_{i1} \geq \cdots \geq w_{is}$.
\end{proof}

Proposition \ref{prop:successive} applies to all models seen in
Subsection~\ref{subsec41} since the $u_i$ are order ideals.

\begin{example} \rm
Consider the two-parameter radial distortion model for two cameras derived in
(\ref{eq:2param}). The vectors in the above proof are
$v = (0,0,1,0,0,1,0,0,1)$ and
$w = \bigl(0,0, (0,0), 0,0, (0,0) ,1,1,(1,1) \bigr)$.
The decomposition (\ref{eq:CSXdecomp}) holds
for all four models $X = E,F,G,G'$.
The penultimate column of Table \ref{tab:2param} says that 
the degree of $(X_{[v]})_{[w]}$ is
bounded above by $12 \cdot {\rm deg}(X)$. This follows directly from
Proposition ~\ref{prop:degdis} because $12 = |v| \cdot |w|$.
\hfill $\diamondsuit$
\end{example}

The exact degrees for $X_{[u]}$ shown in Tables \ref{tab:2param} 
and \ref{tab:4param}  were found using Gr\"obner bases. 
This computation starts from the ideal of $X$ and
incorporates the structure in
Proposition~\ref{prop:successive}.

\smallskip

{\em Tropical Geometry} \cite{MS} furnishes tools for
studying multi-parameter  distortion varieties.
In what follows,  we identify any variety $X \subset \PP^n$ with its reembedding into $\PP^N$,
where the $i$-th coordinate $x_i$ has been duplicated $|u_i|$ times.
Consider the distortion variety  ${\bf 1}_{[u]}$ of the  point ${\bf 1} = (1:1:\cdots:1)$ in $\PP^n$.
This is the toric variety in $\PP^N$ given~by the parametrization
$$ \bigl(  \lambda^{u_{0,1}} \! :  \lambda^{u_{0,2}} \! : \cdots:  \lambda^{u_{0,s_0}} \,:\,
  \lambda^{u_{1,1}} \! :   \cdots:  \lambda^{u_{1,s_1}} \,:\, \cdots \,:\,
 \lambda^{u_{r,1}} \! : \cdots:  \lambda^{u_{r,s_r}} \! \bigr)
\,\,\,\,{\rm for} \,\,\,\,
\lambda \,\in\, (\CC^*)^{r+1}. $$
\noindent Let $\tilde u$ denote the $(r{+}1) \times (N{+}1)$-matrix whose columns are
vectors in the sets $u_i$ for $i=0,1,\ldots,n$, augmented by
an extra all-one row vector $(1,1,\ldots,1)$. This matrix
represents the toric variety ${\bf 1}_{[u]}$.
Recall that the {\em Hadamard product} $\star$ of
two vectors in $\CC^{n+1}$ is their coordinate-wise product. This
operation extends to points in $\PP^n$ and also to subvarieties.

\begin{theorem}
\label{thm:tropicalization}
Fix a projective variety $X \subset \PP^n$ and any distortion system $u$, regarded as $r \times(N+1)$-matrix.
The distortion variety is the Hadamard product of $X$ with a toric variety:
$$ X_{[u]} \,\,=\,\, X \star {\bf 1}_{[u]} $$
Its tropicalization is the Minkowski sum of the tropicalization of $X$ with a linear space:
\begin{equation}
\label{eq:tropicalization}
 {\rm trop}(X_{[u]}) \,\, = \,\, {\rm trop}(X) + {\rm trop}({\bf 1}_{[u]}) \,\, = \,\,
 {\rm trop}(X) + {\rm rowspace}(\tilde u). 
 \end{equation}
\end{theorem}

\begin{proof}
This follows from equation (\ref{eq:uparamulti}) and \cite[\S 5]{MS}.
The toric variety ${\bf 1}_{[u]}$ in $\PP^N$ is represented by the matrix $\tilde u$, in the sense of
\cite{GBCP}, so its tropicalization is the row space of $\tilde u$.
Tropicalization takes Hadamard products into Minkowski sums, by
\cite[Prop.~5.1]{BCK}  or
\cite[Prop.~5.5.11]{MS}. 
\end{proof}

Theorem~\ref{thm:tropicalization} suggests the following method
for computing degrees of multi-parameter distortion varieties.
Let $L$ be the standard tropical linear space of codimension $r+\dim(X)$
in $\mathbb{R}^{N+1}/\mathbb{R} {\bf 1}$, as in \cite[Corollary 3.6.16]{MS}.
 Fix a general point $\xi$ in $\mathbb{R}^{N+1}/\mathbb{R} {\bf 1}$.
 Then ${\rm deg}(X_{[u]})$ is the number of points, counted with multiplicity, 
in the intersection of the tropical variety (\ref{eq:tropicalization})
with the tropical linear space  $\xi + L$.
In practise, $X$ is fixed and we  precompute ${\rm trop}(X)$.
That fan then gets intersected with $\xi+L + {\rm rowspace}(\tilde u)$ 
for various configurations $u$.

\begin{corollary}
The degree of $X_{[u]}$ is a piecewise-linear function 
in the maximal minors of~$\tilde u$.
\end{corollary}

\begin{proof}
The maximal minors of $\tilde u$ are the Pl\"ucker coodinates of the row space of $\tilde u$.
An argument as in \cite[\S 4]{CCDRS} leads to
 a polyhedral chamber decomposition of the relevant Grassmannian, according to which pairs
of cones in ${\rm trop}(X)$ and in $\xi+L+{\rm rowspace}(\tilde u)$ actually intersect. Each such
intersection is a point, and its multiplicity is one of the maximal minors of $\tilde u$.
\end{proof}

\begin{table}[htp]
\small
\centering
\begin{tabular}{|r||c|c|c|c|c|c|c|c|}
\hline Variety $X$ &
$\!{\rm dim} \!$ &lineality& f-vector & multiplicities \\
\hline 
$F$ in Example \ref{ex:intoP14}  & 7 &  4 & (9, 18, 15) & $1_{15} $ \\
\hline
$E$ in Example \ref{ex:E} &  5 & 0 & (591, 4506, 12588, 15102, 6498) & 
$ 2_{6426}, 4_{72} $ \\
\hline
$\! G$ in Example \ref{ex:G}  &    6 & 1 & (32, 213, 603, 780, 390) & $1_{336}, 2_{54}$ \\
\hline
$\! G'$ in Example \ref{ex:Gprime}  &   6 & 1 & (100, 746, 2158, 2800, 1380) & $1_{800}, 2_{572}, 4_8 $  \\
\hline \end{tabular}
\caption{The tropical varieties in $\mathbb{R}^9/\mathbb{R} {\bf 1}$ associated with the two-view models.}
\label{tab:tropical}
\end{table}

Using the software {\tt Gfan} \cite{gfan}, we precomputed the tropical varieties ${\rm trop}(X)$
for our four basic two-view models, namely  $X = E,F,G,G'$. The results are summarized in
Table \ref{tab:tropical}.

The lineality space corresponds to a torus action on $X$. 
Its dimension is given in column~2.
Modulo this space, ${\rm trop}(X)$ is a pointed fan. Column 3 
records the number of $i$-dimensional cones for $i=1,2,3,\ldots$.
Each maximal cone comes with an integer multiplicity \cite[\S 3.4]{MS}.
These multiplicities are $1$, $2$ or $4$ for our examples.
Column 4 indicates their distribution.

\section{Application to Minimal Problems}
\label{sec5}
This section offers a case study for one \textit{minimal problem}
which has not yet been treated in the computer vision literature.
We build and test an efficient Gr\"obner basis solver for it.
Our approach follows  \cite{Kukelova-ECCV-2008, Kukdiss} and
applies in principle to any zero-dimensional  parameterized polynomial system.
This illustrates how the theory in Sections \ref{sec2}, \ref{sec3}, \ref{sec4}
ties in with practise.

We fix the distortion variety f+E+$\lambda$ in Table~\ref{tab:comparison}.
This is the variety $G''_{[v]}$ which lives in  $ \PP^{11}$  and has
dimension $7$ and degree $23$. We represent its defining equations by the matrix
\begin{equation}
\label{eq:fivebythree}
\begin{pmatrix}
\,x_{11} & \, x_{12} & \phantom{-}x_{21} x_{31}+x_{22} x_{32}+x_{23} x_{33} & \, x_{13} \, & y_{13} \,\,\\
\,x_{21} & \, x_{22}  & -x_{11} x_{31}-x_{12} x_{32}-x_{13} x_{33} &  \, x_{23} \, & y_{23} \,\,\\
\,x_{31} & \, x_{32}  & 0 & \, x_{33} \,& y_{33}\,\,
\end{pmatrix}.
\end{equation}
This matrix is derived by augmenting (\ref{eq:threebyfour}) with the $y$-column.
The prime ideal of $G''_{[v]}$ is generated by  all $3 \times 3$-minors  of (\ref{eq:fivebythree})
and the $2 \times 2$-minors in the last two columns.
The real points on this projective variety represent the relative position of two cameras, one with an
unknown focal length $f$, and the other with an unknown radial distortion parameter $\lambda$.

Each pair $(U_1,U_2)$
of image points gives a constraint (\ref{eq:0=x2TiKTEiKx1})
which translates into a linear equation (\ref{eq:cdotm})
on  $G''_{[v]} \cap L' \subset \PP^{11}$.
Here   $\V{m}^\top =  \left[
x_{11},x_{12}, x_{13}, y_{13},
x_{21},x_{22}, x_{23}, y_{23},
x_{31}, x_{32}, x_{33}, y_{33}
\right] $ is the vector of unknowns. Using notation
as in Subsection \ref{sec21}, the coefficient vector  
of the equation $\,\V{c}^\top\V{m} = 0\,$ is $\,
\V{c}^\top = 
\left[
u_2u_1,  u_2v_1,  u_2,  u_2 \|U_1 \|^2,
v_2u_1, v_2v_1, v_2, v_2 \|U_1 \|^2,
u_1, v_1, 1, \| U_1 \|^2
\right] $.

Seven  pairs determine a linear system 
$\M{C}\,\V{m}  \,\, =  \,\,0$ where 
the coefficient matrix $\M{C}$ has format $7 \times 12$.
 For general data, the matrix $\M{C}$ has full rank $7$.
The solution set is a $5$-dimensional linear subspace in $\RR^{12}$,
or, equivalently, a $4$-dimensional subspace $L'$ in $ \PP^{11}$.
The intersection $G''_{[v]} \cap L'$ consists of $23$ points. Our aim is to compute 
these fast and accurately. This is what is meant by the {\em minimal problem}
associated with the distortion variety $G''_{[v]}$.

\subsection{First build elimination template, then solve instances very fast}

We shall employ the method of {\em automatic generation of 
Gr\"obner solvers}. This  has already been applied with considerable success 
to a wide range of camera geometry problems in computer vision;
see e.g~\cite{Kukelova-ECCV-2008, Kukdiss}.
We start by computing a suitable basis 
$\{\V{n}_1,\V{n}_2,\V{n}_3,\V{n}_4,\V{n}_5\}$
for the null space of $\M{C}$ in $\RR^{12}$. We then introduce four unknowns
$\gamma_1,\ldots,\gamma_4$, and we substitute
\begin{eqnarray}
	\label{eq:null}
	\V{m} \,= \,
	 \gamma_1 \V{n}_1 + 
 \gamma_2 \V{n}_2 + 
  \gamma_3 \V{n}_3 + 
   \gamma_4 \V{n}_4 + 
    \V{n}_5 .
\end{eqnarray}
Our rank constraints on (\ref{eq:fivebythree}) translate
into ten equations in 
$\gamma_1,\gamma_2,\gamma_3,\gamma_4$. 
This system has $23$ solutions in $\CC^4$.
Our aim is to compute these within a few tens or hundreds of \underbar{microseconds}.

Efficient and stable Gr\"obner solvers are often based on {\em Stickelberger's Theorem} 
\cite[Theorem 2.6]{cbms},
which expresses the solutions as the joint eigenvalues of its
companion matrices. Let $I \subset \RR[\gamma]$ be the ideal  generated by our ten polynomials in  $ \gamma = (\gamma_1,\gamma_2,\gamma_3,\gamma_4)$.
The quotient ring $ \RR[\gamma]/I$ is isomorphic to $\RR^{23}$.
An $\RR$-vector space basis $B$ is given by the standard monomials
with respect to any Gr\"obner basis of $I$.
The multiplication map $M_i :  \RR[\gamma]/I \rightarrow  \RR[\gamma]/I$,
$f \mapsto f\gamma_i$ is $\RR$-linear.
Using the basis $B$, this becomes a $23 \times 23$-matrix.
The matrices $M_1,M_2,M_3,M_4$ commute pairwise. These 
are the {\em companion matrices}. As an $\RR$-algebra, 
$  \RR[M_1,M_2,M_3,M_4] \simeq \RR[\gamma]/I$.
Since $I$ is radical, there are $23$ linearly independent joint eigenvectors ${\bf x}$,
satisfying $M_i {\bf x} = \lambda_i {\bf x}$. 
The vectors $(\lambda_1,\lambda_2,\lambda_3,\lambda_4) \in \CC^4$
are the zeros of~$I$.

In practise, it suffices to  construct only one of the companion matrices $M_i$,
since we can recover  the  zeros of~$I$ from  eigenvectors ${\bf x}$ of  $M_i$.
Thus, our   primary   task   is    to   compute   either
$M_1,M_2,M_3$ or $M_4$   from   seven  point  correspondences
$(U_1,U_2)$ in a manner that is both very fast and numerically
stable. For this purpose, the {\em automatic generator} of 
Gr\"obner solvers~\cite{Kukelova-ECCV-2008, Kukdiss} is used.
We now explain this method and illustrate it for the f+E+$\lambda$ problem. 
  
 To achieve speed in computation, we exploit  that, for generic data, the
   Buchberger's algorithm always rewrites the input  polynomials in the  same way.
      The resulting Gr\"obner trace~\cite{Traverso-SAC-2005} is always the same.
     Therefore, we can construct a single trace for
  all generic systems  by tracing the construction  of a Gr\"obner
  basis of a single ``generic'' system.  This is done only once in
  an  \textit{off-line}   stage  of   solver  generation. It produces
  an \textit{elimination template}, which is  then reused again and again
    for efficient \textit{on-line}  computations on generic data.
  
  The \textit{off-line}  part of the solver generation is  a variant of the
  Gr\"obner  trace algorithm in~\cite{Traverso-SAC-2005}.
  Based on   the    F4  algorithm \cite{Faugere-JPAA-1999}   for a  particular  generic
  system, it produces an elimination template for constructing a Gr\"obner basis of  $\langle F \rangle$.
    The input polynomial system $F = \{ f_1,\ldots,f_{10}\}$  is written in the form $A\,m =0$, where $A$ is the 
  matrix of coefficients and $m$ is the vectors of monomials of the system. 
  Every Gr\"obner basis $G$ of $F$ can be constructed by Gauss-Jordan (G-J) elimination
   of a coefficient matrix $A_d$   derived from $F$ by multiplying each polynomial $f_i \in F$,
    by all monomials up to degree  $\max\left\{0,d-d_i\right\}$, where $d_i = \deg(f_i)$.

    To find an appropriate $d$, our solver generator starts with $d=\min\left\{d_i\right\}$, sets $m_d=m$, and G-J eliminates the matrix $A_{\min\left\{d_i\right\}}=A$. Then, it checks if a Gr\"obner basis $G$ has been generated. If not,     it increases $d$ by one, 
builds the next    $A_d$ and    $m_d$,
    and goes back to the check. This is repeated until a
    suitable $d$ and a Gr\"obner basis $G$ has been found.
    Often, we can remove some rows (polynomials) from $A_d$ at this stage 
    and form a smaller elimination template, denoted $A_d'$.
    For  this,   another heuristic optimization procedure is employed, 
   aimed at removing unnecessary polynomials and provide an efficient template leading from $F$ to 
   the reduced coefficient matrix $A_d'$.
     For a detailed description 
          see     \cite{Kukelova-ECCV-2008} and     \cite[Section  4.4.3]{Kukdiss}. 

In order to guide this process, we first precompute
  the reduced Gr\"obner basis of $I$, e.g.~w.r.t.\ grevlex ordering in {\tt Macaulay2}~\cite{M2},
  and the associated monomial basis $B$ of $\RR[\gamma]/I$.
  This has to be done in exact arithmetic over $\QQ$,
  which is computationally very  demanding,   due   to    the   coefficient
  growth~\cite{Arnold-JSC-2003}. We alleviate this problem   by  using   modular
  arithmetic~\cite{Faugere-JPAA-1999} or  by computing directly  in a  finite field  modulo   
  a  single ``lucky     prime  number''~\cite{Traverso-SAC-2005}.       For     many      practical
  problems~\cite{Byrod-CVPR-2008,Nister-PAMI-2004,Stewenius-CVPR-2005},
small   primes like 30011 or  30013 are sufficient.

    The output of this off-line algorithm is the elimination template for constructing
       $A_d^{\prime}$, i.e.\ the list of monomials multiplying each polynomial of $F$ to produce $A_d^{\prime}$ and $m_d^{\prime}$.
       The template is encoded as manipulations of sparse coefficient matrices. 
       After removing unnecessary rows and columns,
         the matrix $A_d^{\prime}$ has  size $s \times (s + |B|)$ for some $s$. The left $s \times s$-block is invertible. Multiplying
 $A_d^{\prime}$ by that inverse and extracting appropriate rows, one obtains the $|B| \times |B|$ matrix $M_1$
that represents the linear map  $ \RR[\gamma]/I \rightarrow  \RR[\gamma]/I, f \mapsto f \gamma_1$ 
in the basis $B$.

We applied  this off-line algorithm to the  f+E+$\lambda$ problem, with standard monomial basis
  $$ \begin{matrix}
  B \,\,=\,\, (1, \gamma_1, \gamma_1\gamma_3, \gamma_1\gamma_3\gamma_4, \gamma_1\gamma_4, \gamma_1\gamma_4^2, \gamma_2, \gamma_2\gamma_3, \gamma_2\gamma_3\gamma_4, \gamma_2\gamma_4, \gamma_2\gamma_4^2, \gamma_2\gamma_4^3, \gamma_3, \gamma_3^2, \gamma_3^3,
\\
\gamma_3^2\gamma_4, \gamma_3\gamma_4, \gamma_3\gamma_4^2, \gamma_3\gamma_4^3, \gamma_4, \gamma_4^2, \gamma_4^3, \gamma_4^4 ).
\end{matrix}
$$
Note that $|B| = 23$. The matrix (\ref{eq:fivebythree}) gives the
following ten ideal generators (with $d_1{=}d_2{=}d_3 {=} 2,
d_4{=}d_5{=}3, d_6 {=} \cdots {=} d_{10}  {=} 4$)
 for the variety $G''_{[u]}$  
encoding the f+E+$\lambda$ problem:
$$
\begin{matrix}
f_1&=& y_{23}x_{33}-x_{23}y_{33} \\
f_2&=& y_{13}x_{33}-x_{13}y_{33} \\
f_3&=& y_{13}x_{23}-x_{13}y_{23}\\
f_4&=& y_{13}x_{22}x_{31}-x_{12}y_{23}x_{31}-y_{13}x_{21}x_{32}+x_{11}y_{23}x_{32}+x_{12}x_{21}y_{33}-x_{11}x_{22}y_{33}\\
f_5&=& x_{13}x_{22}x_{31}-x_{12}x_{23}x_{31}-x_{13}x_{21}x_{32}+x_{11}x_{23}x_{32}+x_{12}x_{21}x_{33}-x_{11}x_{22}x_{33}\\
f_6&=& x_{11}y_{13}x_{31}x_{32}+x_{21}y_{23}x_{31}x_{32}+x_{12}y_{13}x_{32}^2+x_{22}y_{23}x_{32}^2-x_{11}x_{12}x_{31}y_{33}-x_{21}x_{22}x_{31}y_{33}\\
& &- x_{12}^2x_{32}y_{33}+x_{13}^2x_{32}y_{33}-x_{22}^2x_{32}y_{33}+x_{23}^2x_{32}y_{33}-x_{12}x_{13}x_{33}y_{33}-x_{22}x_{23}x_{33}y_{33}\\
& & \cdots \qquad \qquad \cdots \qquad \qquad \cdots \qquad \qquad \cdots \qquad \qquad \cdots  \\
f_{10}&=& x_{11}x_{12}x_{31}^2+x_{21}x_{22}x_{31}^2-x_{11}^2x_{31}x_{32}+x_{12}^2x_{31}x_{32}-x_{21}^2x_{31}x_{32}+x_{22}^2x_{31}x_{32}\\
&& \!\! - x_{11}x_{12}x_{32}^2-x_{21}x_{22}x_{32}^2+x_{12}x_{13}x_{31}x_{33}+x_{22}x_{23}x_{31}x_{33}-x_{11}x_{13}x_{32}x_{33} - x_{21}x_{23}x_{32}x_{33}
\end{matrix}
$$
Using (\ref{eq:null}), these are inhomogeneous 
polynomials in $\gamma_1,\gamma_2,\gamma_3,\gamma_4$.
In the off-line algorithm, we multiply $f_i$ by all monomials up to degree $5-d_i$ in these four variables.
Each of $f_1,f_2,f_3$ is multiplied by the $35$ monomials of degree $\leq 3$,
each of $f_4,f_5$ is multiplied by the $15$ monomials of degree $\leq 2$,
and each of $f_6,\ldots,f_{10}$ is multiplied by the $5$ monomials of degree $\leq 1$.
The resulting $160=10+105+30+25$ polynomials are written as a matrix $A_5$ with $160$ rows.
Only $103$ rows are needed to construct the matrix $M_1$. We conclude with an
 elimination template matrix $A_5^{\prime}$ of format  $103 \times  126$. For any data $C$,
 the on-line solver performs G-J elimination  on that  matrix, and it
computes the eigenvectors of a $23  \times 23$  matrix $M_1$.
 
  To  avoid   coefficient  growth   in  the  on-line   stage,  exact
  computations  over $\QQ$   are replaced  by approximate  computations  with  floating  point  numbers
  in $\RR$.  
  In a naive implementation, expected  cancellations may fail to occur due to
  rounding errors, thus leading to incorrect results.  This is not a problem 
    in our method because we follow the  precomputed elimination template:
    we use only matrix entries that were non-zero in the off-line stage. Still, replacing the 
    symbolic F4 algorithm with
    a numerical computation  may   lead    to   very    unstable  behavior. 
  
     It    has   been
  observed~\cite{Bujnak-PhD-2012}  that  different formulations,  term
  orderings, pair selection strategies, etc., can  have a dramatic effect
  on the stability and speed of the final solver. It is hence crucial to 
  validate  every solver  experimentally, by  simulations as well as on real data.

\subsection{Computational results}
\label{sec:synth}

A \textit{complete} solution, in the \textit{engineering sense},
to a minimal problem is a solution that is: 1) \underbar{fast} and 
2) \underbar{numerically stable} for most of the data
  that occur in practice. 
Moreover, for applications it is important to study the distribution of real solutions of the minimal solver. 

Minimal solvers are  often used inside RANSAC style loops~\cite{RANSAC}.
They form parts of  much larger systems, such as structure-from-motion and 3D reconstruction pipelines or localization systems. Maximizing the efficiency of these solvers is an essential task.
  Inside a RANSAC loop, all real zeros returned by the solver are seen as possible solutions to the problem.   The consistency w.r.t.\ all measurements is tested for each of them.
  Since that test may be computationally expensive, the study of the distribution of real solutions is important.

In this section we present graphs and statistics that display properties of the complete 
solution we offer for the f+E+$\lambda$ problem.  We  studied  the  performance of our
Gr\"obner solver on synthetically generated 3D scenes  with known ground-truth parameters.
We  generated  500,000  different   scenes  with  3D  points  randomly
distributed  in  a  cube  $[{-10},10]^{3}$  and  cameras  with  random
feasible  poses.  Each  3D point  was projected  by two  cameras. The
focal length $f$ of the left camera was drawn uniformly from the interval
$\left[0.5,2.5\right]$ and the  focal length of the right camera was
set to $1$.   The
orientations and positions  of the cameras were selected  at random so
as to look at the scene from  a random distance, varying from 20 to 40
from the center  of the scene.  Next, the image  projections in the right
camera  were  corrupted by  random  radial  distortion, following  the
one-parameter division model in~\cite{Fitzgibbon-CVPR-2001}.  The radial
distortion $\lambda$ was drawn uniformly from the interval $\left[-0.7,0\right]$.
The aim was to investigate  the behavior of the  algorithms for large as  well as
small amounts of radial distortion.

\paragraph{Computation and  its speed.}
 The proposed f+E+$\lambda$  solver performs the following steps:
 \begin{enumerate}
    \item  Fill  the  $103 \times 126$  elimination template matrix $A_5^{\prime}$   with  coefficients 
    derived from the input measurements. \vspace{-0.08in}
    \item Perform G-J elimination on the matrix $A_5^{\prime}$. \vspace{-0.08in}
        \item  Extract  the  desired coefficients  from  the  eliminated matrix. \vspace{-0.08in}
    \item Create the multiplication matrix from extracted coefficients. \vspace{-0.08in}
    \item Compute the eigenvectors of the multiplication matrix. \vspace{-0.08in}
     \item          Extract        $23$  complex      solutions          
       $(\gamma_1,\gamma_2,\gamma_3,\gamma_4)$ from the eigenvectors. \vspace{-0.08in}
     \item   For  each real solution    $\left(\gamma_1,\gamma_2,\gamma_3,\gamma_4\right)$,
       recover  the monomial vector   $\V{m}$ as in (\ref{eq:null}),
      the fundamental  matrix $\M{F}$, the  focal  length  $f$, and  the radial
       distortion      $\lambda$.
\end{enumerate}
All seven steps were implemented efficiently.
 The final f+E+$\lambda$  solver runs in less than $1ms$.

\label{sec:nstab}
\begin{figure}[h]
\centering
\begin{tabular}{cc}
\includegraphics[width=0.46\linewidth]{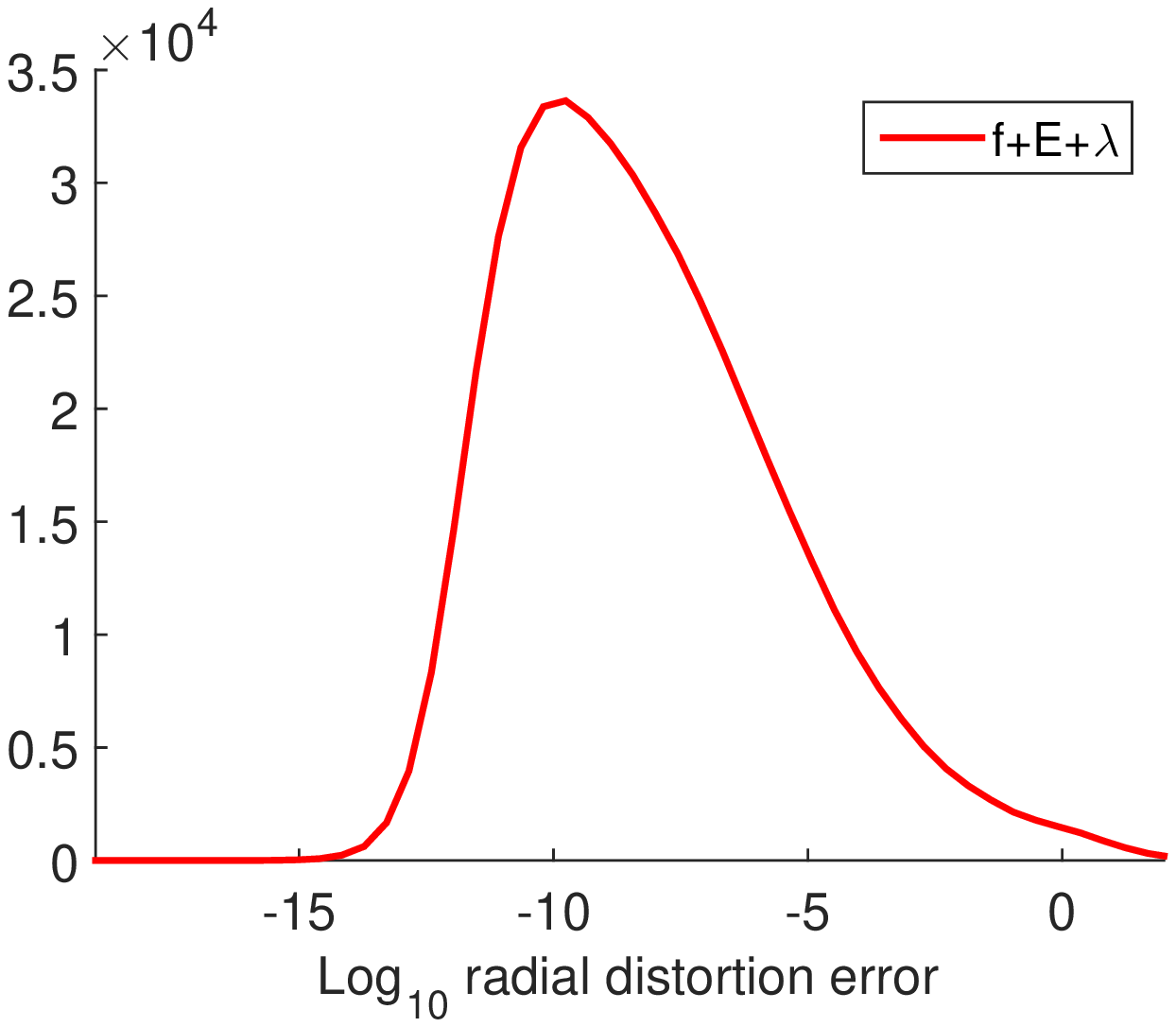}&\quad
\includegraphics[width=0.46\linewidth]{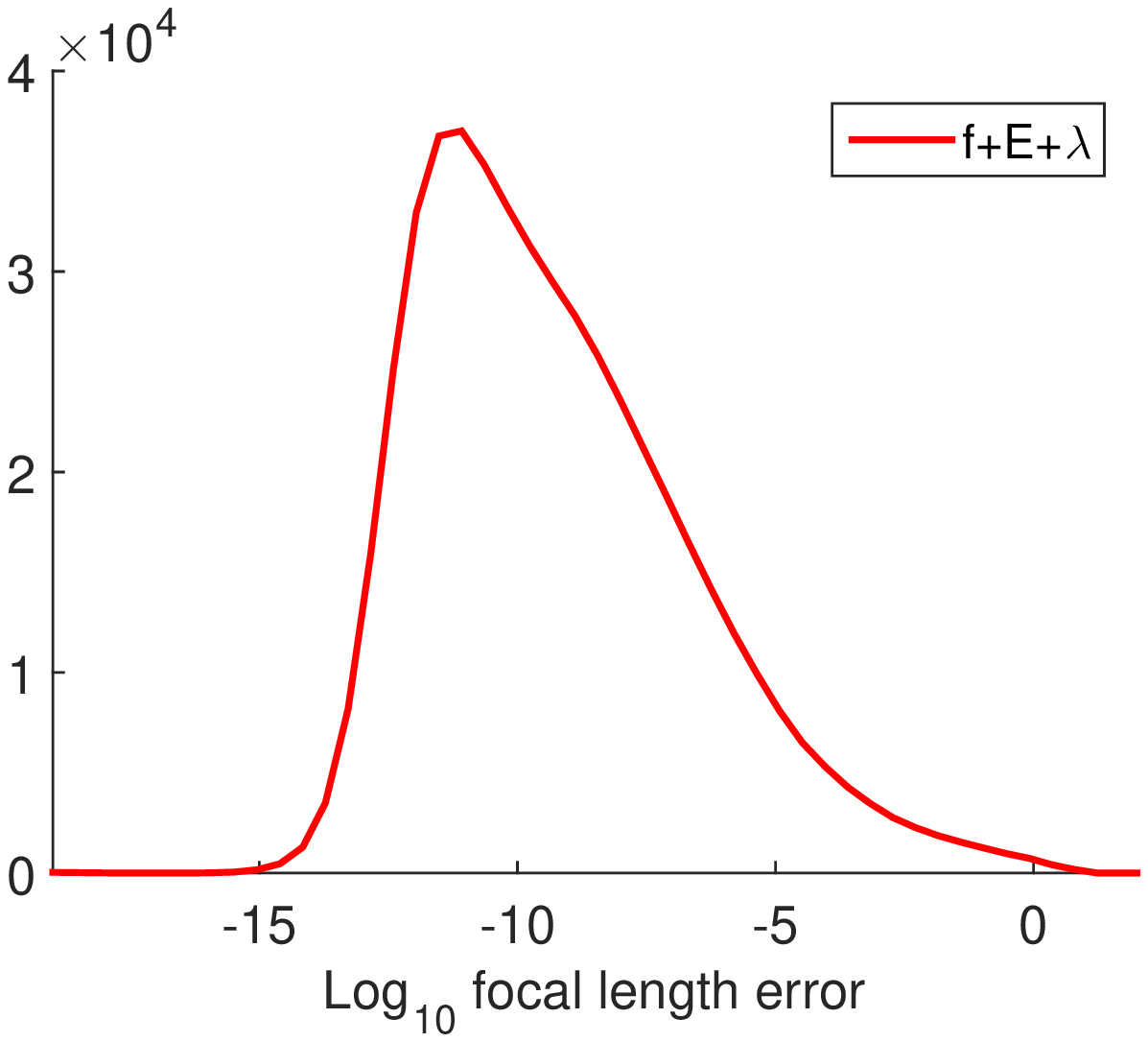}\\
(a)&(b) 
\end{tabular}
\caption{Numerical stability.  (a)  $\textrm{Log}_{10}$ of  the
 relative   error   of   the   estimated   radial   distortion.   (b)
  $\textrm{Log}_{10}$  of the  relative error  of the  estimated focal
  length.}
\label{fig:stability}
\end{figure}

\paragraph{Numerical stability.}
\noindent  
We studied  the behavior  of our solver on noise-free data.
Figure~\ref{fig:stability}(a) shows the experimental frequency of the base 10 logarithm of
the  relative  error  of  the radial  distortion  parameter  $\lambda$
estimated  using  the  new  f+E+$\lambda$ solver.  These  result  were
obtained  by selecting  the real  roots  closest to  the ground  truth
values. The results suggest that the   solver delivers
correct solutions  and its  numerical stability  is suitable  for real
word applications.

Figure~\ref{fig:stability}(b) shows the distribution of $\textrm{Log}_{10}$  of
the relative error of the  estimated focal length $f$.  Again these result
were obtained by selecting the real  roots closest to the ground truth
values. 
 Note  that the f+E+$\lambda$ solver  does not directly compute 
the focal length $f$. Its output is the monomial vector in $\V{m}$~(\ref{eq:null}),
from which we extract $\lambda$ and the  fundamental matrix  $\M{F} =
(x_{ij})$. To obtain the unknown focal  length from $\M{F}$, we use
   the  following formula:
   
\begin{lemma}
  Let $\V{X} = (x_{ij})_{1 \leq i, j \leq 3}$ be a generic point in the variety $G''$ from Example \ref{ex:more}.  Then there are exactly two pairs of essential matrix and focal length $(\V{E},f)$ such that $\V{X} = \textup{diag}(f^{-1},f^{-1},1) \V{E}$.  If one of them is $(\V{E},f)$ then the other is $(\textup{diag}(-1,-1,1)\V{E}, \, -f)$.  In particular, f is determined up to sign by $\V{X}$.  A formula to recover $f$ from $\V{X}$ is as follows:
\begin{eqnarray}
\label{eq:focal}
    f^2 \,\,= \,\, \frac{ x_{23}x_{31}^2+x_{23}x_{32}^{2}-2x_{21}x_{31}x_{33}-2x_{22}x_{32}x_{33}-x_{23}x_{33}^{2}}{2x_{11}x_{13}x_{21}+2x_{12}x_{13}x_{22}-x_{11}^{2}x_{23}-x_{12}^{2}x_{23}+x_{13}^{2}x_{23}+x_{21}^{2}x_{23}+x_{22}^{2}x_{23}+x_{23}^3}.
\end{eqnarray}
\end{lemma}

\begin{proof}
Consider the map $E \times \CC^{*} \rightarrow \PP^8$, $(\V{E}, f) \mapsto \textup{diag}(f^{-1},f^{-1},1) \V{E}$.
Let $I \subset \QQ[e_{ij}, f, x_{ij}]$ be the ideal of the graph of this map.  So, $I$ is generated by the ten Demazure cubics and the nine entries of $\V{X} - \textup{diag}(f^{-1},f^{-1},1) \V{E}$.  
We computed the elimination ideal $I \cap \QQ[f, x_{ij}]$ in \texttt{Macaulay2}.
The polynomial gotten by clearing the denominator and subtracting the RHS from the LHS in the formula
(\ref{eq:focal})  lies in this elimination ideal.
This proves the lemma. \end{proof}

\begin{figure}[h]
\centering
\begin{tabular}{cc}
\includegraphics[width=0.45\linewidth]{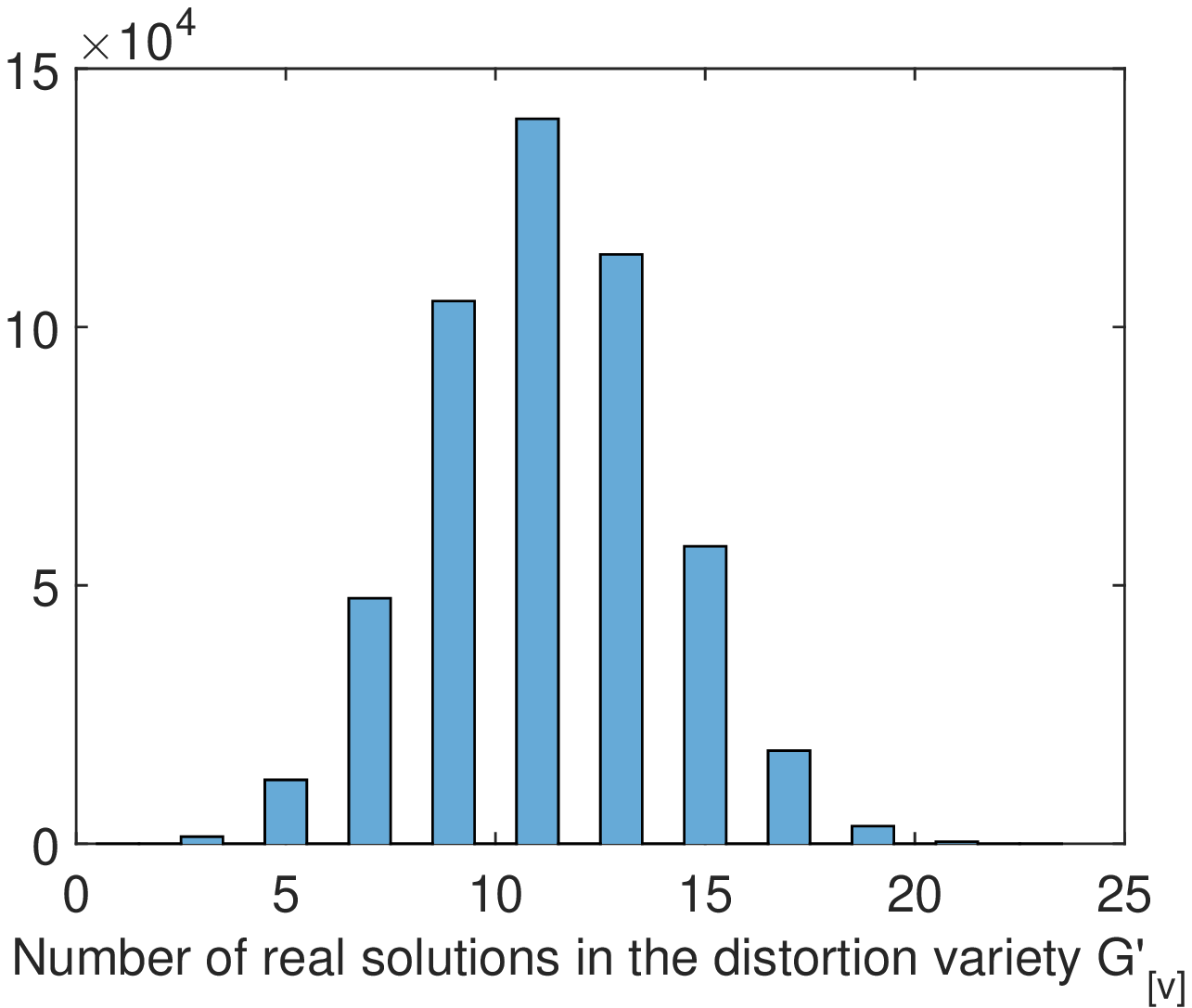}& \quad
\includegraphics[width=0.45\linewidth]{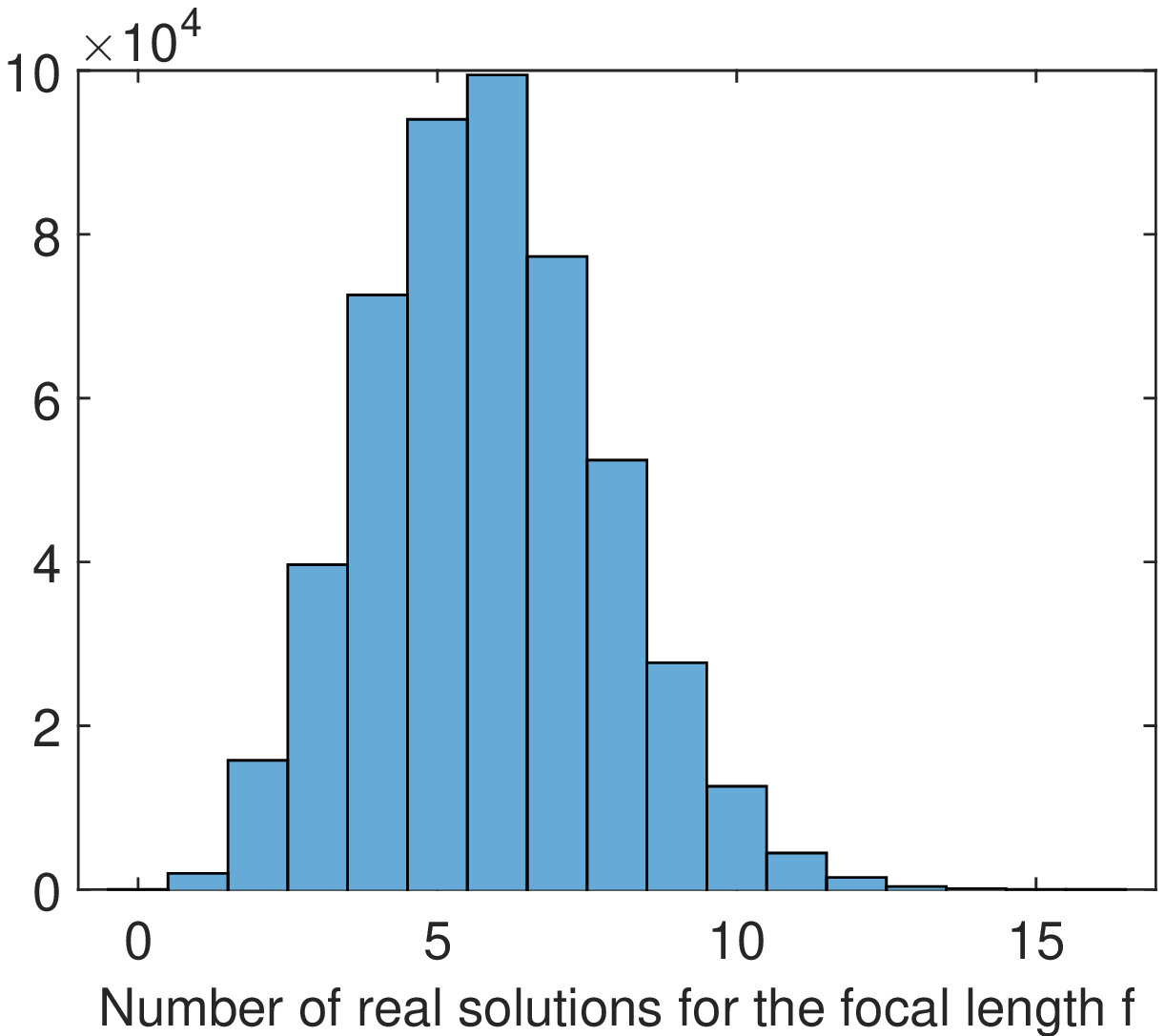}\\
(a) & (b) 
\end{tabular}
\caption{Number  of real solutions for  floating point computation
  with noise-free image data.}
\label{fig:hist1}
\end{figure}

\paragraph{Counting real solutions.}\label{sec:r_sol}
In the  next experiment we studied  the distribution of the  number of
real solutions $(\lambda,F)$ and the number of real
solutions for the focal length $f$.

Figure~\ref{fig:hist1} (a) shows  the histogram of the  number of real
solutions on the  distortion variety $G''_{[v]}$.  All
odd integers between $1$  and $23$  were observed. Most of  the time we got
an odd  number of real solutions  between $7$ and $15$.  The empirical probabilities
 are in Table~\ref{tab:real_sol_var}.
\begin{table}[h]
\footnotesize
\centering
\begin{tabular}{|c||c|c|c|c|c|c|c|c|c|c|c|c|}
\hline
real roots &  & & & & & & & & & & & \\
in $G''_{[v]}$ &  1 & 3 & 5 & 7 & 9 & 11 & 13 & 15 & 17 & 19 & 21 & 23  \\
\hline 
$\%$ & 0.003 & 0.276 & 2.47 & 9.50 & 21.0 & 28.0 & 22.8 & 11.5& 3.60 & 0.681 & 0.078 & 0.003\\
\hline 
\end{tabular}
\caption{Percentage of the number  of real solutions in the distortion variety $G''_{[v]}$.}
\label{tab:real_sol_var}
\end{table}

Figure~\ref{fig:hist1} (b) shows  the histogram of the  number of 
solutions for  the focal length $f$, computed from the
distortion  variety  $G''_{[v]}$  using  the  formula~(\ref{eq:focal}).
Of the $46$ complex solutions, at most $23$ could be real and positive. The largest number
of positive real solutions $f$ observed in  in 500,000 runs  was $16$.
The empirical probabilities
from this experiment are in Table~\ref{tab:real_sol_f}.
\begin{table}[htp]
\footnotesize
\centering
\begin{tabular}{|c||c|c|c|c|c|c|c|c|c|c|c|c|c|c|c|c|}\hline
real $f$ &  0 & 1 & 2 & 3 & 4 & 5 & 6 & 7 & 8 &  9 & 10 & 11 \\\hline 
$\%$ &  0.003 & 0.397 & 3.16 & 7.93 & 14.5 & 18.8 & 19.9 & 15.5 & 10.5 & 5.54 & 2.52 & 0.894\\\hline\hline
real $f$ & 12 & 13 & 14 & 15 & 16 &&&&&&&\\\hline
$\%$ & 0.295 &0.075   & 0.023  &  0.005 &  0.001 &&&&&&& \\\hline
\end{tabular}
\caption{Percentage of  the number of positive real  roots for the
    focal length $f$.}
\label{tab:real_sol_f}
\end{table}

We  performed  the same  experiment with  image  measurements
corrupted by Gaussian noise  with the standard  deviation set  to 2
pixels.   
The distribution of the real roots in the  distortion variety $G''_{[v]}$ 
was very similar to the distribution for noise-free data. 
The main difference  between these result and those
 for noise-free data was in the number of real values  for the
focal length  $f$. For  a fundamental matrix  corrupted by  noise, the
formula~(\ref{eq:focal})  results in  no  real  solutions more  often.
See Tables \ref{tab:real_sol_var2} and \ref{tab:real_sol_f2} for the
empirical probabilities.

\begin{table}[h]
\footnotesize
\centering
\begin{tabular}{|r||c|c|c|c|c|c|c|c|c|c|c|c|}\hline
real roots  &  1 & 3 & 5 & 7 & 9 & 11 & 13 & 15 & 17 & 19 & 21 & 23\\\hline 
$\%$ & 0.021 &  0.509 & 3.23 & 11.2 & 22.4 & 27.7 & 21.1 & 10.1 & 3.07 & 0.566 & 0.062 & 0.004\\\hline
\end{tabular}
\caption{Percentage of  the  number   of  real  solutions  in  the
  distortion variety $G''_{[v]}$ for  image measurements corrupted with
  Gaussian noise with $\sigma = 2$ pixels.
\label{tab:real_sol_var2}}
  \bigskip
 \end{table}
 \begin{table}[h]
\footnotesize
\centering
\begin{tabular}{|c||c|c|c|c|c|c|c|c|c|c|c|c|}\hline
real $f$ & 0 & 1 & 2 & 3 & 4 & 5 & 6 & 7 & 8 & 9 & 10 & 11 \\\hline 
$\%$ & 0.243 & 1.30 & 4.92 & 10.2 & 16.1 & 19.0 & 18.5 & 13.7 & 8.79 & 4.33 & 1.96 & 0.689\\\hline\hline
real $f$ & 12 & 13 & 14 & 15 & 16 &&&&&&&\\\hline
$\%$ & 0.217 & 0.048 & 0.015 & 0.002 & 0.001 &&&&&&&\\\hline
\end{tabular}
\caption{Percentage of the number  of real roots for the focal
  length $f$ with data as in Table~\ref{tab:real_sol_var2}.
\label{tab:real_sol_f2}}
\end{table}

Finally, we  performed the  same experiments for a  special camera
motion. 
It  is known~\cite{Newsam-1996, Sturm-2001}  
that  the   focal  length  cannot  be
determined  by the formula~(\ref{eq:focal}) from the fundamental  matrix if the  optical axes
are parallel to each other, e.g.\ for a sideways motion of cameras.
Therefore, we generated cameras  undergoing
``close-to-sideways  motion''.    To model this
scenario,   100   points   were   again    placed   in   a   3D   cube
$\left[-\textrm{10},\textrm{10}\right]^\textrm{3}$.     Then   500,000
different camera pairs were generated such that both cameras
were  first  pointed   in  the  same  direction   (optical  axes  were
intersecting at infinity) and then translated laterally. Next, a small
amount of  rotational noise  of 0.01 degrees  was introduced  into the
camera  poses   by  right-multiplying   the  projection   matrices  by
respective rotation matrices. This multiplication slightly rotated the
optical axes of  cameras (as not to intersect at  infinity) as well as
simultaneously displaced the camera  centers.


The results for noise-free  data are 
displayed in         Tables~\ref{tab:real_sol_var_sideways}
and~\ref{tab:real_sol_f_sideways}.  For  this
special  close-to-sideways motion, 
the formula (\ref{eq:focal}) provides up to $20$ real solutions
for  the focal length~$f$.

\begin{table}[h]
\footnotesize
\centering
\begin{tabular}{|r||c|c|c|c|c|c|c|c|c|c|c|c|}\hline
real roots &  1 & 3 & 5 & 7 & 9 & 11 & 13 & 15 & 17 & 19 & 21 & 23\\\hline 
$\%$ & 0.007 & 0.544 & 5.14 & 16.83 & 26.2 & 24.9 & 16.2 & 7.37 & 2.30 & 0.475 & 0.061 & 0.006\\\hline
\end{tabular}
\caption{Real solutions  in  the
distortion   variety   $G''_{[v]}$   for  the close-to-sideways   motion
 scenario.
\label{tab:real_sol_var_sideways}}
\smallskip
\footnotesize
\centering
\begin{tabular}{|c||c|c|c|c|c|c|c|c|c|c|c|c|}\hline
real $f$ & 0 & 1 & 2 & 3 & 4 & 5 & 6 & 7 & 8 & 9 & 10 & 11 \\\hline 
$\%$ & 0.006 & 0.755 & 3.08 & 10.2 & 12.9 & 20.9 & 16.2 & 16.0 & 8.73 & 6.17 & 2.61 & 1.58\\\hline\hline
real $f$ & 12 & 13 & 14 & 15 & 16 & 17 & 18 & 19 & 20 &&&\\\hline
$\%$ & 0.556 & 0.253 & 0.086 & 0.033 & 0.011 & 0.0044 & 0.0016 & 0.0012 & 0.0002 &&&\\\hline
\end{tabular}
\caption{\small{Real solutions for the focal length $f$ in the close-to-sideways motion scenario.}
\label{tab:real_sol_f_sideways}}
\end{table}

\bigskip

\begin{small}
\noindent
{\bf Acknowledgement.}\smallskip \\
 This project started at the {\em Algebraic Vision} workshop (May 2016)
at the American  Institute of Mathematics (AIM) in  San Jose. 
 We are grateful to the organizers,  Sameer Agarwal, Max Lieblich and Rekha Thomas, 
 for bringing us together.  Joe Kileel and Bernd Sturmfels were supported by the
 US National Science Foundation  (DMS-1419018). 
Zuzana Kukelova was supported by the Czech Science Foundation (GACR P103/12/G08).
Part of this study was carried out while she worked for Microsoft Research, Cambridge, UK.
Tomas Pajdla was supported by H2020-ICT-731970 LADIO.
 \end{small}

\medskip

\begin{small}

\end{small}

\bigskip \medskip \bigskip

\noindent
\footnotesize {\bf Authors' addresses:}

\smallskip

\noindent Joe Kileel and Bernd Sturmfels,
University of California, Berkeley, USA,
{\tt \{jkileel,bernd\}@berkeley.edu}

\noindent Zuzana Kukelova and Tomas Pajdla,
Czech Technical University in Prague,
{\tt \{kukelova,pajdla\}@cvut.cz}

\end{document}